\documentclass[12pt,reqno]{amsart}

\usepackage[cp1251]{inputenc}
\usepackage{amsthm,amssymb,amsmath,amsfonts}
\usepackage{graphicx}
\usepackage{mathabx}
\usepackage[matrix,arrow,curve]{xy}
\newcommand{\Z}{\ensuremath{\mathbb{Z}}}

\newcommand{\A}{\ensuremath{\mathbb{A}}}

\newcommand{\F}{\ensuremath{\mathbb{F}}}

\newcommand{\CG}{\ensuremath{\mathfrak{C}}}

\newcommand{\DG}{\ensuremath{\mathfrak{D}}}

\newcommand{\AG}{\ensuremath{\mathfrak{A}}}

\newcommand{\SG}{\ensuremath{\mathfrak{S}}}

\newcommand{\VG}{\ensuremath{\mathfrak{V}}}

\newcommand{\ka}{\ensuremath{\Bbbk}}

\newcommand{\kka}{\ensuremath{\overline{\Bbbk}}}

\newcommand{\XX}{\ensuremath{\overline{X}}}

\newcommand{\K}{\ensuremath{\mathrm{K}}}

\newcommand{\Pro}{\ensuremath{\mathbb{P}}}

\newcommand{\Aut}{\ensuremath{\operatorname{Aut}}}

\newcommand{\Gal}{\ensuremath{\operatorname{Gal}}}

\newcommand{\Pic}{\ensuremath{\operatorname{Pic}}}

\newcommand{\ord}{\ensuremath{\operatorname{ord}}}

\newcommand{\mmatr}[4]{\ensuremath{\left(
\begin{array}{ccc} #1&#2\\#3&#4\\
\end{array}
\right)}}

\makeatletter
\@addtoreset{equation}{section}
\makeatother

\newtheorem{theorem}[equation]{Theorem}
\newtheorem{proposition}[equation]{Proposition}
\newtheorem{lemma}[equation]{Lemma}
\newtheorem{corollary}[equation]{Corollary}

\theoremstyle{definition}
\newtheorem{example}[equation]{Example}
\newtheorem{definition}[equation]{Definition}

\theoremstyle{remark}
\newtheorem{remark}[equation]{Remark}
\newtheorem*{notation}{Notation}

\title{Quotients of del Pezzo surfaces \\ of high degree}
\address{Institute for Information Transmission Problems, 19 Bolshoy Karetnyi side-str., Moscow 127994, Russia}
\address{Laboratory of Algebraic Geometry, National Research University Higher School of Economics, 6 Usacheva str., Moscow 119048, Russia}
\email{trepalin@mccme.ru}
\thanks{The author of article was supported by the Russian Academic Excellence Project '5--100', Young Russian Mathematics award, and the grant RFFI 15-01-02164-a}
\author{Andrey Trepalin}

\begin{document}

\begin{abstract}
In this paper we study quotients of del Pezzo surfaces of degree four and more over arbitrary field $\ka$ of characteristic zero by finite groups of automorphisms. We show that if a del Pezzo surface $X$ contains a point defined over the ground field and the degree of $X$ is at least five then the quotient is always $\ka$-rational. If the degree of $X$ is equal to four then the quotient can be non-$\ka$-rational only if the order of the group is $1$, $2$ or $4$. For these groups we construct examples of non-$\ka$-rational quotients.
\end{abstract}

\maketitle
\section{Introduction}

In this paper we study question about rationality of quotients of del Pezzo surfaces over arbitrary field $\ka$ of characteristic zero by finite groups of automorphisms. We say that a surface $S$ is \textit{$\ka$-rational} if there exists a birational map $S \dashrightarrow \Pro^2_{\ka}$ defined over $\ka$. If for the algebraic closure $\kka$ of $\ka$ such a map defined over $\kka$ exists for a surface $\overline{S} = S \otimes_{\ka} \kka$ and $\Pro^2_{\kka}$, we say that $S$ is \textit{rational}. Note that in many other papers for these notions the authors use terms \textit{rational surface} and \textit{geometrically rational surface} respectively. 

Let $\ka$ be any field of characteristic zero. We want to know when quotients of $\ka$-rational surfaces by finite groups are $\ka$-rational. From results of the $G$-equivariant minimal model program we know that any quotient of a $\ka$-rational surface is birationally equivalent to a quotient of a conic bundle or a del Pezzo surface by the same group (see \cite[Theorem~1]{Isk79}).

In \cite{Tr16} it was shown that non-$\ka$-rational quotients of $\ka$-rational surfaces form a birationally unbounded family. But all examples considered in \cite{Tr16} are quotients of conic bundles.

In \cite{Tr14} it was shown that any quotient of the projective plane (which is a del Pezzo surface of degree $9$) is $\ka$-rational. In this paper we consider quotients of del Pezzo surfaces of degree no less than $4$. We show that if the set of $\ka$-points on the surface is non-empty then its quotient is $\ka$-rational except for a small number of cases. The main result of the paper is the following.

\begin{theorem}
\label{main}
Let $\ka$ be a field of characteristic zero, $X$ be a del Pezzo surface over $\ka$ such that $X(\ka) \ne \varnothing$ and $G$ be a finite subgroup of automorphisms of $X$. If $K_X^2 \geqslant 5$ then the quotient variety $X / G$ is $\ka$-rational. If $K_X^2 = 4$, the order of $G$ is equal to $1$, $2$, or $4$, and all nontrivial elements of $G$ do not have curves of fixed points, then $X / G$ can be not $\ka$-rational. In all other possibilities of $G$ if $K_X^2 = 4$ then $X / G$ is $\ka$-rational.
\end{theorem}

For a surface $X$ admitting a structure of conic bundle such that $K_X^2 \geqslant 5$ and $X(\ka) \ne \varnothing$ the quotient $X / G$ is $\ka$-rational for any finite subgroup $G \subset \Aut(X)$ by Theorem \cite[Proposition 1.6]{Tr16}. Therefore we have the following corollary. 

\begin{corollary}
\label{geq5}
Let $\ka$ be a field of characteristic zero, $X$ be a smooth rational surface over $\ka$ such that $X(\ka) \ne \varnothing$ and $G$ be a finite subgroup of automorphisms of $X$. If $K_X^2 \geqslant 5$ then the quotient variety $X / G$ is $\ka$-rational.
\end{corollary}

Note that a minimal del Pezzo surface $X$ of degree $4$ such that $X(\ka) \ne \varnothing$ is not \mbox{$\ka$-rational} by Iskovskikh rationality criterion (see \cite[Chapter 4]{Isk96}). This gives us an example of a del Pezzo surface of degree $4$ such that its quotient by the trivial group is not $\ka$-rational. For groups $G$ of order $2$ and $4$ we explicitely construct examples such that $X$ is $G$-minimal and $\ka$-rational and $X / G$ is not $\ka$-rational. Also for these groups we construct examples of non-$\ka$-rational quotients of non-$\ka$-rational del Pezzo surfaces of degree $4$, $\ka$-rational quotients of non-$\ka$-rational del Pezzo surfaces of degree $4$ and $\ka$-rational quotients of $G$-minimal $\ka$-rational del Pezzo surfaces of degree $4$.

To prove Theorem \ref{main} we consider case-by-case del Pezzo surfaces of degree from $9$ to~$4$ and study their quotients by finite groups. The cases of degree $9$ and $6$ are considered in~\cite{Tr14} (see Theorems \ref{DP9} and \ref{DP6} below). The cases of degree $8$, $5$ and $4$ are considered in Propositions \ref{DP8}, \ref{DP5} and \ref{DP4} respectively. If degree is $7$ then a del Pezzo surface is never $G$-mininal, and its quotient is birationally equivalent to a quotient of a del Pezzo surface of degree $8$ or $9$. Therefore we do not consider this case.

The plan of this paper is as follows. In Section 2 we review some notions and facts about minimal rational surfaces, groups, singularities and quotients. In Section 3 we study quotients of del Pezzo surfaces of degree $8$ and show that they are all $\ka$-rational. In Section 4 we study quotients of del Pezzo surfaces of degree $5$ and show that they are all $\ka$-rational. In Section 5 we show that quotients of del Pezzo surfaces of degree $4$ are $\ka$-rational for all nontrivial groups except three cases. In Section 6 we show that for the remaining three cases the quotient of a del Pezzo surface of degree $4$ can be non-$\ka$-rational and give explicit examples of non-$\ka$-rational quotients of $\ka$-rational surfaces.

The author is grateful to his adviser Yu.\,G.\,Prokhorov and to C.\,A.\,Shramov for useful discussions and to I.\,V.\,Netay for his help in finding explicit equations of $(-1)$-curves in Examples \ref{DP4i12ex} and \ref{DP4C2i15ex}.

\begin{notation}

Throughout this paper $\ka$ is any field of characteristic zero, $\kka$ is its algebraic closure. For a surface $X$ we denote $X \otimes \kka$ by $\XX$. For a surface $X$ we denote the Picard group (resp. $G$-invariant Picard group) by $\Pic(X)$ (resp. $\Pic(X)^G$). The number \mbox{$\rho(X) = \operatorname{rk} \Pic(X)$} (resp. \mbox{$\rho(X)^G = \operatorname{rk} \Pic(X)^G$}) is the Picard number (resp. the $G$-invariant Picard number) of $X$. If two surfaces $X$ and $Y$ are $\ka$-birationally equivalent then we write $X \approx Y$. If two divisors $A$ and $B$ are linearly equivalent then we write $A \sim B$. The rational ruled (Hirzebruch) surface $\Pro_{\Pro^1}\left( \mathcal{O} \oplus \mathcal{O}(n) \right)$ is denoted by $\F_n$.

\end{notation}

\section{Preliminaries}

\subsection{$G$-minimal rational surfaces}

In this subsection we review main notions and results of $G$-equivariant minimal model program following the papers \cite{Man67}, \cite{Isk79}, \cite{DI1}. Throughout this subsection $G$ is a finite group.

\begin{definition}
\label{rationality}
A {\it rational variety} $X$ is a variety over $\ka$ such that $\XX=X \otimes \kka$ is birationally equivalent to $\Pro^n_{\kka}$.

A {\it $\ka$-rational variety} $X$ is a variety over $\ka$ such that $X$ is birationally equivalent to $\Pro^n_{\ka}$.

A variety $X$ over $\ka$ is a {\it $\ka$-unirational variety} if there exists a $\ka$-rational variety $Y$ and a dominant rational map $\varphi: Y \dashrightarrow X$.
\end{definition}

\begin{definition}
\label{minimality}
A {\it $G$-surface} is a pair $(X, G)$ where $X$ is a projective surface over $\ka$ and~$G$ is a finite subgroup of $\Aut_{\ka}(X)$. A morphism of $G$-surfaces $f: X \rightarrow X'$ is called a \textit{$G$-morphism} if for each $g \in G$ one has $fg = gf$.

A smooth $G$-surface $(X, G)$ is called {\it $G$-minimal} if any birational morphism of smooth $G$-surfaces $(X, G) \rightarrow (X',G)$ is an isomorphism.

Let $(X, G)$ be a smooth $G$-surface. A $G$-minimal surface $(Y, G)$ is called a {\it minimal model} of $(X, G)$ or {\it $G$-minimal model} of $X$ if there exists a birational \mbox{$G$-morphism $X \rightarrow Y$}.
\end{definition}

The following theorem is a classical result about the $G$-equivariant minimal model program.

\begin{theorem}
\label{GMMP}
Any birational $G$-morphism $f:X \rightarrow Y$ of smooth $G$-surfaces can be factorized in the following way:
$$
X= X_0 \xrightarrow{f_0} X_1 \xrightarrow{f_1} \ldots \xrightarrow{f_{n-2}} X_{n-1} \xrightarrow{f_{n-1}} X_n = Y,
$$
where each $f_i$ is a contraction of a set $\Sigma_i$ of disjoint $(-1)$-curves on~$X_i$, such that $\Sigma_i$ is defined over $\ka$ and $G$-invariant. In particular,
$$
K_Y^2 - K_X^2 \geqslant \rho(X)^G - \rho(Y)^G.
$$
\end{theorem}

The classification of $G$-minimal rational surfaces is well-known due to V.\,Iskovskikh and Yu.\,Manin (see \cite{Isk79} and \cite{Man67}). We introduce some important notions before surveying it.

\begin{definition}
\label{Cbundledef}
A smooth rational $G$-surface $(X, G)$ admits a {\it conic bundle} structure if there exists a $G$-equivariant map $\varphi: X \rightarrow B$ such that any scheme fibre is isomorphic to a reduced conic in~$\Pro^2_{\ka}$ and $B$ is a smooth curve.
\end{definition}

\begin{definition}
\label{DPdef}
A {\it del Pezzo surface} is a smooth projective surface~$X$ such that the anticanonical class $-K_X$ is ample.

A {\it singular del Pezzo surface} is a normal projective surface $X$ such that the anticanonical class $-K_X$ is ample and all singularities of $X$ are Du Val singularities.

The number $d = K_X^2$ is called the {\it degree} of a (singular) del Pezzo surface $X$.
\end{definition}

A del Pezzo surface $\XX$ over $\kka$ is isomorphic to $\Pro^2_{\kka}$, $\Pro^1_{\kka} \times \Pro^1_{\kka}$ or a blowup of $\Pro^2_{\kka}$ at up to 8 points in general position (see \cite[Theorem~2.5]{Man74}). The configuration of $(-1)$-curves on a del Pezzo surface plays an important role in studying its geometry. Throughout this paper we will use the notation from the following remark.

\begin{remark}
\label{DP_1curves}
Let $\XX$ be a del Pezzo surface of degree $d$, $4 \leqslant d \leqslant 7$. Then $\XX$ can be realised as a blowup $f: \XX \rightarrow \Pro^2_{\kka}$ at $n = 9 - d$ points $p_1$, $\ldots$, $p_n$ in general position. Put $E_i = f^{-1}(p_i)$ and $L = f^*(l)$, where~$l$ is the class of a line on $\Pro^2_{\kka}$. One has
$$
-K_{\XX} \sim 3L - \sum \limits_{i=1}^n E_i.
$$
The $(-1)$-curves on $\XX$ are $E_i$, the proper transforms \mbox{$L_{ij} \sim L - E_i - E_j$} of the lines passing through a pair of points $p_i$ and $p_j$, and for $d = 4$, the proper transform
$$
Q \sim 2L - \sum \limits_{i = 1}^5 E_i
$$
\noindent of the conic passing through the five points $p_1$, $p_2$, $p_3$, $p_4$, $p_5$.

In this notation one has:
$$
E_i \cdot E_j = 0; \qquad E_i \cdot L_{ij} = 1; \qquad E_i \cdot L_{jk} = 0;
$$
$$
L_{ij} \cdot L_{ik} = 0; \qquad L_{ij} \cdot L_{kl} = 1;
$$
$$
E_i \cdot Q = 1; \qquad L_{ij} \cdot Q = 0,
$$
\noindent where $i$, $j$, $k$ and $l$ are different numbers from the set $\{1, 2, 3, 4, 5\}$.
\end{remark}

\begin{theorem}[{\cite[Theorem 1]{Isk79}}]
\label{Minclass}
Let $X$ be a $G$-minimal rational $G$-surface. Then either $X$ admits a $G$-equivariant conic bundle structure with $\Pic(X)^{G} \cong \Z^2$, or $X$ is a del Pezzo surface with $\Pic(X)^{G} \cong \Z$.
\end{theorem}

\begin{theorem}[{cf. \cite[Theorem 4]{Isk79}, \cite[Theorem 5]{Isk79}}]
\label{MinCB}
Let $X$ admit a $G$-equivariant conic bundle structure, $G \subset \Aut(X)$. Then:

(i) If $K_X^2 = 3, 5, 6, 7$ or $X \cong \F_1$ then $X$ is not $G$-minimal.

(ii) If $K_X^2 = 8$ then $X$ is isomorphic to $\F_n$, and $X$ is $G$-minimal if $n \ne 1$.

(iii) If $K_X^2 \ne 3,5,6,7,8$ and $\rho(X)^G = 2$ then $X$ is $G$-minimal.

\end{theorem}

The following theorem is an important criterion of $\ka$-rationality over an arbitrary perfect field $\ka$.

\begin{theorem}[{\cite[Chapter 4]{Isk96}}]
\label{ratcrit}
A minimal rational surface $X$ over a perfect field $\ka$ is $\ka$-rational if and only if the following two conditions are satisfied:

(i) $X(\ka) \ne \varnothing$;

(ii) $K_X^2 \geqslant 5$.
\end{theorem}

An important class of rational surfaces is the class of toric surfaces.

\begin{definition}
\label{tordef}
A \textit{toric variety} is a normal variety over $\ka$ containing an algebraic torus as a Zariski dense subset, such that the action of the torus on itself by left multiplication extends to the whole variety.

A variety $X$ is called a \textit{$\ka$-form of a toric variety} if $\XX$ is toric.
\end{definition}

Obviously, a $\ka$-form of a toric variety is rational.

The following lemma is well-known (see, for example, \mbox{\cite[Lemma~2.9]{Tr14}}).

\begin{lemma}
\label{torcrit}
Let $X$ be a $G$-minimal rational surface such that $X(\ka) \ne \varnothing$. The following are equivalent:

(i) $X$ is a $\ka$-form of a toric surface;

(ii) $K_X^2 \geqslant 6$;

(iii) $X$ is isomorphic to $\Pro^2_{\ka}$, a smooth quadric $Q \subset \Pro^3_{\ka}$, a del Pezzo surface of degree $6$ or a minimal rational ruled \mbox{surface $\F_n$ ($n \geqslant 2$)}.
\end{lemma}

\begin{corollary}
\label{piccrit}
Let $X$ be a smooth rational $G$-surface such that $X(\ka) \ne \varnothing$ and \mbox{$\rho(X)^G + K_X^2 \geqslant 7$}. Then there exists a $G$-minimal model $Y$ of $X$ such that $Y$ is a $\ka$-form of a toric surface. In particular, $X$ is $\ka$-rational.
\end{corollary}
\begin{proof}
By Theorem \ref{Minclass} there exists a birational $G$-morphism \mbox{$f: X \rightarrow Z$} such \mbox{that $\rho(Z)^G \leqslant 2$}. By Theorem \ref{GMMP} one has
$$
K_Z^2 \geqslant K_X^2 + \rho(X)^G - \rho(Z)^G \geqslant 7 - \rho(Z)^G.
$$

If $\rho(Z)^G = 1$ then $K_Z^2 \geqslant 6$ and $Z$ is a $\ka$-form of a toric surface by Lemma \ref{torcrit}. In this case we put $Y = Z$.

If $\rho(Z)^G = 2$ and $K_Z^2 = 5$ then $Z$ is not $G$-minimal by Theorem \ref{MinCB}(i). Therefore there exists a minimal model $Y$ of $Z$ such that $K_Y^2 \geqslant 6$ and $Y$ is a $\ka$-form of a toric surface by Lemma \ref{torcrit}.

The set $X(\ka)$ is not empty. Thus $Y(\ka) \ne \varnothing$ and $X \approx Y$ is $\ka$-rational by Theorem~\ref{ratcrit}.
\end{proof}

\subsection{Groups}

In this subsection we collect some results and notation concerning groups used in this paper.

We use the following notation:

\begin{itemize}

\item $\CG_n$ denotes a cyclic group of order $n$;

\item $\DG_{2n}$ denotes a dihedral group of order $2n$;

\item $\SG_n$ denotes a symmetric group of degree $n$;

\item $\AG_n$ denotes a alternating group of degree $n$;

\item $(i_1 i_2 \ldots i_j)$ denotes a cyclic permutation of $i_1$, \ldots, $i_j$;

\item $\VG_4$ denotes a Klein group isomorphic to $\CG_2^2$;

\item $\langle g_1, \ldots, g_n \rangle$ denotes a group generated by $g_1$, \ldots, $g_n$;

\item $A${\SMALL$\bullet$}$B$ is an extension of $B$ by $A$, i.\,e. if $G \cong A${\SMALL$\bullet$}$B$ then there exists an exact sequence:
$$
1 \rightarrow A \rightarrow G \rightarrow B \rightarrow 1;
$$

\item for surjective homomorphisms $\alpha: A \rightarrow D$ and $\beta: B \rightarrow D$ we denote by $A \triangle_D B$ the diagonal product of $A$ and $B$ over their common homomorphic image $D$ that is the subgroup of $A \times B$ of pairs $(a; b)$ such that $\alpha(a) = \beta(b)$;

\item $\operatorname{diag}(a, b) = \mmatr{a}{0}{0}{b}$;

\item $i = \sqrt{-1}$;

\item $\xi_n = e^{\frac{2\pi i}{n}}$;

\item $\omega = \xi_3 = e^{\frac{2\pi i}{3}}$.

\end{itemize}

To find fixed points of groups acting on a del Pezzo surface of degree~$8$ we apply the following well-known lemma. For the proof see, for example, \cite[Lemma 3.4]{Tr16}.

\begin{lemma}
\label{fixedpoints}
Elements $g_1, g_2 \in \mathrm{PGL}_2 \left( \kka \right)$ such that the group $H = \langle g_1, g_2\rangle$ is finite have the same pair of fixed points on $\Pro^1_{\kka}$ if and only if the group $H$ is cyclic. Otherwise the elements $g_1$ and $g_2$ do not have a common fixed point.
\end{lemma}

The group $\SG_5$ often appears as a group of automorphisms of a rational surface. Therefore it is important to know its subgroups and normal subgroups of these subgroups. The following lemma is an easy exercise.

\begin{lemma}
\label{S5subgroups}
Any nontrivial subgroup $G \subset \SG_5$ contains a normal subgroup $N$ conjugate in $\SG_5$ to one of the following groups:

\begin{itemize}

\item $\CG_2 = \langle(12)\rangle$,
\item $\CG_2 = \langle(12)(34)\rangle$,
\item $\CG_3 = \langle(123)\rangle$,
\item $\VG_4 = \langle(12)(34), (13)(24)\rangle$,
\item $\CG_5 = \langle(12345)\rangle$,
\item $\AG_5$.

\end{itemize}

\end{lemma}

The following lemma is well-known.

\begin{lemma}
\label{C3fixedpoint}
Let a group $\CG_3$ act on $\Pro^2_{\kka}$ and do not have curves of fixed points. Then the group $\CG_3$ has three isolated fixed points and acts on the tangent space of $\Pro^2_{\kka}$ at each fixed point as $\mathrm{diag}(\omega, \omega^2)$.
\end{lemma}

\begin{proof}
The action of any element of finite order $n$ on $\Pro^2_{\kka}$ can be diagonalized in a way such that the entries of the diagonal matrix corresponding to this element are roots of unity of the $n$-th degree. Therefore one can choose coordinates on $\Pro^2_{\kka}$ in which the action of $\CG_3$ on $\Pro^2_{\kka}$ has form~$\mathrm{diag}\left(\omega^a, \omega^b, \omega^c \right)$, where $a$, $b$, $c \in \{0, 1, 2\}$. If $a = b = c$ then the action is trivial, and if two of these numbers are equal then the group $\CG_3$ has a curve of fixed points. Thus the numbers $a$, $b$ and $c$ are distinct, the group $\CG_3$ has three isolated fixed points
$$
(1 : 0 : 0), \qquad (0 : 1 : 0), \qquad (0 : 0 : 1)
$$
and acts on the tangent space of $\Pro^2_{\kka}$ at each fixed point as $\mathrm{diag}(\omega, \omega^2)$.

\end{proof}

\begin{remark}
\label{C3fixed}
One can check that if we blow up a $\CG_3$-fixed point $p$ such that the group acts on the tangent space at $p$ as $\mathrm{diag}(\omega, \omega^2)$ then on the exceptional divisor there are two fixed points of the group $\CG_3$ and the group acts on the tangent spaces at these points as $\mathrm{diag}(\omega, \omega)$. So starting from $\Pro^2_{\kka}$ we can study actions of $\CG_3$ on del Pezzo surfaces.
\end{remark}

\subsection{Singularities}

In this subsection we review some results about quotient singularities and their resolutions.

All singularities appearing in this paper are toric singularities. These singularities are locally isomorphic to the quotient of $\A^2$ by the cyclic group generated by $\operatorname{diag}(\xi_m, \xi_m^q)$. Such a singularity is denoted by $\frac{1}{m}(1,q)$. If $\gcd(m,q) > 1$ then the group
$$
\CG_m \cong \langle \operatorname{diag}(\xi_m, \xi_m^q) \rangle
$$
\noindent contains a reflection and the quotient singularity is isomorphic to a quotient singularity with smaller $m$.

A toric singularity can be resolved by some weighted blowups. Therefore it is easy to describe numerical properties of a quotient singularity. We list here these properties for singularities appearing in our paper.

\begin{remark}
Let the group $\CG_m$ act on a smooth surface $X$ and \mbox{$f: X \rightarrow S$} be the quotient map. Let $p$ be a singular point on $S$ of type $\frac{1}{m}(1,q)$. Let $C$ and $D$ be curves passing through $p$ such that $f^{-1}(C)$ and $f^{-1}(D)$ are $\CG_m$-invariant and tangent vectors of these curves at the point $f^{-1}(p)$ are eigenvectors of the natural action of $\CG_m$ on $T_{f^{-1}(p)} X$ (the curve $C$~corresponds to the eigenvalue $\xi_m$ and the curve $D$ corresponds to the eigenvalue~$\xi_m^q$).

Let $\pi: \widetilde{S} \rightarrow S$ be the minimal resolution of the singular point $p$. Table \ref{table1} presents some numerical properties of $\widetilde{S}$ and $S$ for the singularities with $m \leqslant 5$.

The exceptional divisor of $\pi$ is a chain of transversally intersecting exceptional curves~$E_i$ whose self-intersection numbers are listed in the last column of Table \ref{table1}. The curves~$\pi^{-1}_*(C)$ and $\pi^{-1}_*(D)$ transversally intersect at a point only the first and the last of these curves respectively and do not intersect other components of exceptional divisor of~$\pi$.

\begin{table}
\caption{} \label{table1}

\begin{tabular}{|c|c|c|c|c|c|}
\hline
$m$ & $q$ & $K_{\widetilde{S}}^2 - K_S^2$ & $\pi^{-1}_*(C)^2 - C^2$ & $\pi^{-1}_*(D)^2 - D^2$ \rule[-7pt]{0pt}{20pt} & $E_i^2$ \\
\hline
$2$ & $1$ & $0$ & $-\dfrac{1}{2}$ & $-\dfrac{1}{2}$ \rule[-11pt]{0pt}{30pt} & $-2$ \\
\hline
$3$ & $1$ & $-\dfrac{1}{3}$ & $-\dfrac{1}{3}$ & $-\dfrac{1}{3}$ \rule[-11pt]{0pt}{30pt} & $-3$ \\
\hline
$3$ & $2$ & $0$ & $-\dfrac{2}{3}$ & $-\dfrac{2}{3}$ \rule[-11pt]{0pt}{30pt} & $-2$, $-2$ \\
\hline
$4$ & $1$ & $-1$ & $-\dfrac{1}{4}$ & $-\dfrac{1}{4}$ \rule[-11pt]{0pt}{30pt} & $-4$ \\
\hline
$4$ & $3$ & $0$ & $-\dfrac{3}{4}$ & $-\dfrac{3}{4}$ \rule[-11pt]{0pt}{30pt} & $-2$, $-2$, $-2$ \\
\hline
$5$ & $1$ & $-\dfrac{9}{5}$ & $-\dfrac{1}{5}$ & $-\dfrac{1}{5}$ \rule[-11pt]{0pt}{30pt} & $-5$ \\
\hline
$5$ & $2$ & $-\dfrac{2}{5}$ & $-\dfrac{2}{5}$ & $-\dfrac{3}{5}$ \rule[-11pt]{0pt}{30pt} & $-3$, $-2$ \\
\hline
$5$ & $3$ & $-\dfrac{2}{5}$ & $-\dfrac{3}{5}$ & $-\dfrac{2}{5}$ \rule[-11pt]{0pt}{30pt} & $-2$, $-3$ \\
\hline
$5$ & $4$ & $0$ & $-\dfrac{4}{5}$ & $-\dfrac{4}{5}$ \rule[-11pt]{0pt}{30pt} & $-2$, $-2$, $-2$, $-2$ \\
\hline
\end{tabular}

\end{table}

\end{remark}

\subsection{Quotients}

In this subsection we collect some results about quotients of rational surfaces.

The following lemma is well-known, see, e.g., \cite[Lemma 4.2]{Tr14}.

\begin{lemma}
\label{toriclemma}
Let $\XX$ be an $n$-dimensional toric variety over a field $\kka$ and let $G$ be a finite subgroup in $\Aut\left(\XX\right)$ conjugate to a subgroup of $n$-dimensional torus $\overline{\mathbb{T}}^n \subset \XX$ acting on~$\XX$. Then the quotient $\XX / G$ is a toric variety.

In particular, if $G$ is a finite cyclic subgroup of the connected component of the \mbox{identity $\Aut^0(\XX) \subset \Aut(\XX)$}, then the quotient $\XX / G$ is a toric variety.
\end{lemma}

We use the following definition for convenience.

\begin{definition}
\label{MMPred}
Let $X$ be a $G$-surface, $\widetilde{X} \rightarrow X$ be its \mbox{($G$-equivariant)} minimal resolution of singularities, and $Y$ be a \mbox{$G$-equivariant} minimal model of $\widetilde{X}$. We call the surface $Y$ a \mbox{\textit{$G$-MMP-reduction}} of $X$.
\end{definition}

Del Pezzo surfaces of degree $8$ considered in this paper are toric surfaces. Thus the following proposition is very useful.

\begin{proposition}[{\cite[Proposition 4.5]{Tr14}}]
\label{cyclictoricquotient}
Let a group $G$ contain a normal subgroup $\CG_p$, where $p$ is prime. If $X$ is a \mbox{$G$-minimal} $\ka$-unirational $\ka$-form of a toric surface then there exists a \mbox{$G / \CG_p$-MMP-reduction} $Y$ of $X / \CG_p$ such that $Y$ is a $\ka$-form of a toric surface. In particular, $X / \CG_p$ is $\ka$-rational.
\end{proposition}

The quotients of del Pezzo surfaces of degree $9$ and $6$ and conic bundles with $K_X^2 \geqslant 5$ were considered in the authors papers \cite{Tr14} and~\cite{Tr16}.

\begin{theorem}[{\cite[Theorem 1.3]{Tr14}}]
\label{DP9}
Let $G \subset \mathrm{PGL}_3(\ka)$ be a finite subgroup. \mbox{Then $\Pro^2_{\ka} / G$} is $\ka$-rational.
\end{theorem}

\begin{theorem}[{\cite[Corollary 1.4]{Tr14}}]
\label{DP6}
Let $X$ be a del Pezzo surface of degree $6$ over $\ka$ such that $X(\ka) \ne \varnothing$ and let $G$ be a finite subgroup of automorphisms of $X$. Then the quotient variety $X / G$ is $\ka$-rational.
\end{theorem}

\begin{theorem}[{\cite[Proposition 1.6]{Tr16}}]
\label{RCbundle}
Let $X$ be a conic bundle such that $K_X^2 \geqslant 5$ \mbox{and $X(\ka) \ne \varnothing$}, and let $G$ be a finite subgroup of $\Aut_{\ka}(X)$. Then $X / G$ is $\ka$-rational.
\end{theorem}

\subsection{Singular del Pezzo surfaces}

In this subsection we explicitly construct $G$-MMP-reductions for some singular del Pezzo surfaces.

\begin{lemma}
\label{DP2sing6A1}
Let a finite group $G$ act on a singular del Pezzo surface $V$ of degree $2$ with six $A_1$ singularities. Then there exists a $G$-MMP-reduction $Y$ of $V$ such that $Y$ is a $\ka$-form of a toric surface.
\end{lemma}

\begin{proof}

For any del Pezzo surface $V$ of degree $2$ with at worst Du Val singularities the linear system $|-K_V|$ is base point free and defines a double cover
$$
f: V \rightarrow \Pro^2_{\ka}
$$
\noindent branched over a reduced quartic $B \subset \Pro^2_{\ka}$. The singularities of $V$ correspond to the singularities of $B$. In our case from the local equations one can obtain that $B$ has six nodes. We are going to show that $B$ is a union of four lines.

Consider a conic $D$ on $\Pro^2_{\kka}$ passing through $5$ of these nodes. Since
$$
D \cdot B = 8 < 10,
$$
\noindent the curves $D$ and $B$ have a common irreducible component. If this component is an irreducible conic then $B$ consists either of two irreducible conics, or of an irreducible conic and two lines. In both cases the number of nodes is less than six. So $B$ consists of a line and a cubic. This cubic has $3$ nodes thus it consists of three lines and $B$ consists of four lines $l_1$, $l_2$, $l_3$ and $l_4$ no three passing through a point.

The preimage $f^{-1}(l_i)$ is a rational curve passing through three singular points. From the Hurwitz formula one has
$$
f^{-1}(l_i) \cdot f^{-1}(l_j) = \frac{1}{2}.
$$
Moreover,
$$
K_V \cdot f^{-1}(l_i) = f^* \left( K_{\Pro^2_{\kka}} + \frac{B}{2} \right) \cdot \frac{f^*(l_i)}{2} = \left( K_{\Pro^2_{\kka}} + \frac{B}{2} \right) \cdot l_i = -1.
$$

Let $\pi: \widetilde{V} \rightarrow V$ be the minimal $G$-equivariant resolution of singularities. Then the proper transform $\pi^{-1}_* f^{-1} (l_1 + l_2 + l_3 + l_4)$ consists of four disjoint \mbox{$(-1)$-curves}, and this quadruple is defined over~$\ka$. We can $G$-equivariantly contract these four curves and get a surface $Y$ with
$$
K_Y^2 = K_{\widetilde{V}}^2 + 4 = K_{V}^2 + 4 = 6.
$$
So $Y$ is a $\ka$-form of a toric surface by Lemma \ref{torcrit}.

\end{proof}

\begin{lemma}
\label{DP1sing2A4}
Let a finite group $G$ act on a singular del Pezzo surface $V$ of degree $1$ with two $A_4$ singularities. Then there exists a $G$-MMP-reduction $Y$ of $V$ such that $\overline{Y} \cong \Pro^2_{\kka}$.

\end{lemma}

\begin{proof}

Let $\pi: \widetilde{V} \rightarrow V$ be the minimal resolution of singularities. The dual graph of curves with negative self-intersection on $\widetilde{V}$ is well-known (see \cite[Table 3, 10e]{AN06}).

\centerline{\includegraphics[scale = 2]{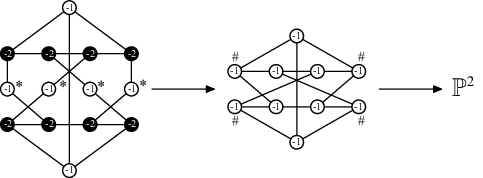}}

Let us equivariantly contract the four disjoint $(-1)$-curves marked by $*$, then equivariantly contract the four $(-1)$-curves marked by $\varhash$ and get a surface $Y$. One has
$$
K_Y^2 = K_{\widetilde{V}}^2 + 8 = K_{V}^2 + 8 = 9.
$$
Therefore $\overline{Y}$ is isomorphic to $\Pro^2_{\kka}$ by Theorem \ref{Minclass}.

\end{proof}

\section{Del Pezzo surface of degree $8$}

In this section we prove the following proposition.

\begin{proposition}
\label{DP8}
Let $X$ be a del Pezzo surface of degree $8$ such that $X(\ka) \ne \varnothing$ and let~$G$ be a finite subgroup of $\Aut_{\ka}(X)$. Then $X / G$ is $\ka$-rational.
\end{proposition}

We start with several auxiliary assertions.

\begin{lemma}
\label{DP8subgroups}
Let $X$ be a del Pezzo surface of degree $8$ such that \mbox{$X(\ka) \ne \varnothing$} and let $G$ be a finite subgroup of $\Aut_{\ka}(X)$. Suppose that $\rho\left(X\right)^G = 1$. Then $X$ is isomorphic to a smooth quadric $Q \subset \Pro^3_{\ka}$ and the group $G$ is isomorphic to $A \triangle_D A$ or $(A \triangle_D A) ${\SMALL$\bullet$}$ \CG_2$ where $A$ is one of the following groups: $\CG_n$, $\DG_{2n}$, $\AG_4$, $\SG_4$ or $\AG_5$, and $D$ is a subgroup of~$A$.
\end{lemma}

\begin{proof}
If $\XX$ is isomorphic to the blowup of $\Pro^2_{\kka}$ at a point then $X$ is not minimal. Therefore~$\XX$ is isomorphic to $\Pro^1_{\kka} \times \Pro^1_{\kka}$ and
$$
\Aut\left(\XX\right) \cong \left( \mathrm{PGL}_2(\kka) \times \mathrm{PGL}_2(\kka) \right) \rtimes \CG_2.
$$
\noindent Let $\pi_1: \XX \rightarrow \Pro^1_{\kka}$ and $\pi_2: \XX \rightarrow \Pro^1_{\kka}$ be the projections on the first and the second factors \mbox{of $\Pro^1_{\kka} \times \Pro^1_{\kka}$} respectively. The group $\Pic\left(\XX\right) \cong \Z^2$ is generated by classes of $a = \pi_1^{-1}(p)$ \mbox{and $b = \pi_2^{-1}(q)$}, where $p$ and $q$ are points on the first and the second factors respectively.

The group
$$
G_0 = G \cap \left( \mathrm{PGL}_2(\kka) \times \mathrm{PGL}_2(\kka) \right)
$$
\noindent preserves the bundles $\pi_1$ and $\pi_2$. Thus $G_0$ naturally acts on the factors of $\Pro^1_{\kka} \times \Pro^1_{\kka}$. Let $A \subset \mathrm{PGL}_2 \left( \kka \right)$ and \mbox{$B \subset \mathrm{PGL}_2 \left( \kka \right)$} be the images of $G_0$ under the projections \mbox{of $\mathrm{PGL}_2(\kka) \times \mathrm{PGL}_2(\kka)$} onto its factors. Then the group $G_0$ is a group $A \triangle_D B$ for some $D$. If the groups $A$ and $B$ are not isomorphic then $G = G_0$, any element $g \in \Gal\left(\kka / \ka\right) \times G$ preserves the factors of $\Pro^1_{\kka} \times \Pro^1_{\kka}$ and one has $ga \sim a$ and $gb \sim b$, so that $\rho\left(X\right)^G = 2$. \mbox{Thus $A \cong B$} and the group $G_0$ is $A \triangle_D A$ for some $D$. Therefore the group $G$ is $A \triangle_D A$ or $(A \triangle_D A) ${\SMALL$\bullet$}$ \CG_2$, where $A$ is a finite subgroup of $\mathrm{PGL}_2(\kka)$, i.e. $\CG_n$, $\DG_{2n}$, $\AG_4$, $\SG_4$ or~$\AG_5$.
\end{proof}

Throughout the rest of this section we use the notation introduced in Lemma \ref{DP8subgroups}.

\begin{lemma}
\label{DP8product}
Let a group $G \cong A \times B$ act on a smooth quadric $X \subset \Pro^3_{\ka}$ such that $X(\ka) \ne \varnothing$, let the group $A$ act trivially on $\pi_2 \left( \XX \right)$ and the group $B$ act trivially on~$\pi_1 \left(\XX \right)$. Then~$X / G$ is isomorphic to a smooth quadric in $\Pro^3_{\ka}$.
\end{lemma}

\begin{proof}
One has
$$
\XX / G \cong \left(\Pro^1_{\kka} \times \Pro^1_{\kka}\right) / \left( A \times B \right) = \left( \Pro^1_{\kka} / A \right) \times \left( \Pro^1_{\kka} / B \right) \cong \Pro^1_{\kka} \times \Pro^1_{\kka}.
$$
Thus $X / G$ is isomorphic to a smooth quadric in $\Pro^3_{\ka}$.
\end{proof}

\begin{lemma}
\label{DP8cyclic}
If a group $G \cong \CG_n \triangle_D \CG_m$ acts on a smooth quadric \mbox{$X \subset \Pro^3_{\ka}$} then $X / G$ is a $\ka$-form of a toric surface.
\end{lemma}

\begin{proof}

The groups $\CG_n$ and $\CG_m$ are subgroups of tori $\mathbb{T}_1 \subset \Aut\left(\pi_1\left(\XX\right)\right)$ \mbox{and $\mathbb{T}_2 \subset \Aut\left(\pi_2\left(\XX\right)\right)$} respectively. One has $\CG_n \times \CG_m \subset \mathbb{T}_1 \times \mathbb{T}_2$. Thus the group $G \cong \CG_n \triangle_D \CG_m$ is a subgroup of the torus $\mathbb{T}_1 \times \mathbb{T}_2 \subset \Aut\left(\XX\right)$. Therefore $X / G$ is a $\ka$-form of a toric surface by Lemma \ref{toriclemma}.

\end{proof}

\begin{remark}
\label{DP8diag}
Let a finite group $H \cong H \triangle_{H} H$ act on a smooth quadric $X \subset \Pro^3_{\ka}$ and faithfully acts on the both factors of \mbox{$\XX \cong \Pro^1_{\kka} \times \Pro^1_{\kka}$}. Then by Lemma \ref{fixedpoints} each cyclic subgroup $C$ of $H$ has four isolated fixed points that are the intersection of $C$-invariant fibres of $\pi_1$ and $\pi_2$. In the neighbourhood of these points the group $C$ acts as $\langle \operatorname{diag}\left(\xi_n, \xi_n^m \right) \rangle$, $\langle \operatorname{diag}\left(\xi_n^{n-1}, \xi_n^m \right) \rangle$, $\langle \operatorname{diag}\left(\xi_n, \xi_n^{n-m} \right) \rangle$ and $\langle \operatorname{diag}\left(\xi_n^{n-1}, \xi_n^{n-m} \right) \rangle$, where $n = \ord H$. If two elements $g_1$, $g_2 \in H$ do not lie in a common cyclic subgroup of $H$ then they do not have common fixed points by Lemma~\ref{fixedpoints}.
\end{remark}

\begin{lemma}
\label{DP8V4}
Let a finite group $G$ act on a smooth quadric $X \subset \Pro^3_{\ka}$ and
$$
N \cong \VG_4 \cong \VG_4 \triangle_{\VG_4} \VG_4
$$
\noindent be a normal subgroup in $G$ acting faithfully on the both factors of \mbox{$\XX \cong \Pro^1_{\kka} \times \Pro^1_{\kka}$}. Then there exists a $G / N$-MMP-reduction $Y$ of $X / N$ such that $Y$ is a $\ka$-form of a toric surface.
\end{lemma}

\begin{proof}

By Remark \ref{DP8diag} there are $12$ points on $\XX$ each of which is fixed by a non-trivial element of $N$, and no curves of fixed points for these elements. Therefore
$$
K_{X / N}^2 = \frac{K_X^2}{4} = 2,
$$
\noindent and the surface $\XX / N$ is a singular del Pezzo surface of degree $2$ with six $A_1$ singularities.

By Lemma \ref{DP2sing6A1} there exists a $G / N$-MMP-reduction $Y$ of $X / N$ such that $Y$ is a $\ka$-form of a toric surface.
\end{proof}

The group $\AG_5$ has two different representations in $\mathrm{PGL}_2\left(\kka\right)$ up to conjugation. Let us consider a group $\AG_5 \cong \AG_5 \triangle_{\AG_5} \AG_5$ acting faithfully on the both factors of $\Pro^1_{\kka} \times \Pro^1_{\kka}$. If the actions of $\AG_5$ on the both factors are conjugate then we call such an action on $\Pro^1_{\kka} \times \Pro^1_{\kka}$ \textit{diagonal}. If the actions of $\AG_5$ on the both factors are not conjugate then we call such an action \mbox{on $\Pro^1_{\kka} \times \Pro^1_{\kka}$} \textit{twisted diagonal}.

\begin{lemma}
\label{DP8A5diag}
Let a finite group $G$ act on a smooth quadric $X \subset \Pro^3_{\ka}$ and
$$
N \cong \AG_5 \cong \AG_5 \triangle_{\AG_5} \AG_5
$$
\noindent be a normal subgroup in $G$ acting faithfully on the both factors of \mbox{$\XX \cong \Pro^1_{\kka} \times \Pro^1_{\kka}$} such that the action is diagonal and $\rho(X)^G = 1$. Then there exists a $G / N$-MMP-reduction $Y$ of $X / N$ such that $Y$ is a $\ka$-form of a toric surface.
\end{lemma}

\begin{proof}

Each nontrivial cyclic subgroup in $\AG_5$ is conjugate to $\langle (12)(34) \rangle$, $\langle (123) \rangle$ or~$\langle (12345) \rangle$.

There are $15$ subgroups conjugate to $\langle (12)(34) \rangle$. By Remark \ref{DP8diag} each of these groups have four fixed points, and act on the neighbourhood of these points as $\langle \operatorname{diag}\left(-1, -1 \right) \rangle$. The stabilizer subgroup of each of these points has order $2$. Therefore there are two $A_1$ singular points on $\XX / N$.

There are $10$ subgroups conjugate to $\langle (123) \rangle$. By Remark \ref{DP8diag} each of these groups have four fixed points, and act on the neighbourhood of these points as $\langle \operatorname{diag}\left(\omega, \omega \right) \rangle$ or $\langle \operatorname{diag}\left(\omega, \omega^2 \right) \rangle$. The stabilizer subgroup of each of these points has order $3$. Therefore there are one $A_2$ singular point and one $\frac{1}{3}(1,1)$ singular point on $\XX / N$.

There are $6$ subgroups conjugate to $\langle (12345) \rangle$. By Remark \ref{DP8diag} each of these groups have four fixed points, and act on the neighbourhood of these points as $\langle \operatorname{diag}\left(\xi_5, \xi_5 \right) \rangle$ or $\langle \operatorname{diag}\left(\xi_5, \xi_5^4 \right) \rangle$ since the action is diagonal. The stabilizer subgroup of each of these points has order $5$. Therefore there are one $A_4$ singular point and one $\frac{1}{5}(1,1)$ singular point on~$\XX / N$.

Hence the set of singular points of $\XX / N$ is the following: two $A_1$ points, one $A_2$ point, one $\frac{1}{3}(1,1)$ point, one $A_4$ point and one $\frac{1}{5}(1,1)$ point. Non-trivial elements of the group $N$ have only isolated fixed points. Thus
$$
K_{X/N}^2 = \frac{K_X^2}{60} = \frac{2}{15}, \quad \rho(X/N)^{G/N} = \rho(X)^G = 1.
$$

Let $f: X \rightarrow X / N$ be the quotient morphism and $r: \widetilde{X / N} \rightarrow X / N$ be the minimal resolution of the singularities, $F_1$ and $F_2$ be a $\CG_5$-invariant fibres of the projections~$\pi_1$ and $\pi_2$ respectively. Note that there exists an element $g$ of the group $G \times \Gal(\kka / \ka)$ such that $gF_1 = F_2$ since $\rho(X)^G = 1$.

One has
$$
K_{\widetilde{X / N}}^2 = \K_{X / N}^2 - \frac{1}{3} - \frac{9}{5} = -2, \quad \rho(\widetilde{X/N})^{G/N} \geqslant \rho(X/N)^{G/N} + 6 = 7.
$$
Moreover the curves $r^{-1}_* f(F_1)$ and $r^{-1}_* f(F_2)$ are two disjoint curves on $\widetilde{\XX / N}$ with self-intersection numbers equal to $-1$ (see Table \ref{table1}). One can $G / N$-equivariantly contract this pair of curves and then $G/N$-equivariantly contract the transforms of two $(-2)$-curves that are the ends of the chain of rational curves in the preimage of the $A_4$ singular point. We obtain a surface $Z$ such that $K_Z^2 = 2$ and $\rho(Z)^{G / N} \geqslant 5$. By Corollary \ref{piccrit} there exists a $G / N$-minimal model $Y$ of $Z$ such that $Y$ is a $\ka$-form of a toric surface.

\end{proof}

\begin{lemma}
\label{DP8A5twisted}
Let a finite group $G$ act on a smooth quadric $X \subset \Pro^3_{\ka}$ and
$$
N \cong \AG_5 \cong \AG_5 \triangle_{\AG_5} \AG_5
$$
\noindent be a normal subgroup in $G$ acting faithfully on the both factors of \mbox{$\XX \cong \Pro^1_{\kka} \times \Pro^1_{\kka}$} such that the action is twisted diagonal and $\rho(X)^G = 1$. Then there exists a $G / N$-MMP-reduction $Y$ of $X / N$ such that~$Y$ is a $\ka$-form of a toric surface.
\end{lemma}

\begin{proof}

As in the proof of Lemma \ref{DP8A5diag} we can apply Remark \ref{DP8diag} and find the list of singularities of $\XX / N$: two $A_1$ points, one $A_2$ point, one $\frac{1}{3}(1,1)$ point and two $\frac{1}{5}(1,2)$ points.

One has
$$
K_{\widetilde{X / N}}^2 = \K_{X / N}^2 - \frac{1}{3} - 2 \cdot \frac{2}{5} = -1, \quad \rho(\widetilde{X/N})^{G/N} \geqslant \rho(X/N)^{G/N} + 5 = 6.
$$
Moreover, for $f$, $r$, $F_1$, $F_2$ defined as in the proof of Lemma \ref{DP8A5diag}, the curves $r^{-1}_* f(F_1)$ and $r^{-1}_* f(F_2)$ are two disjoint curves on $\widetilde{\XX / N}$ with self-intersection numbers equal to~$-1$ (see Table \ref{table1}). One can $G / N$-equivariantly contract this pair of curves and then \mbox{$G/N$-equivariantly} contract the transforms of two $(-2)$-curves which are irreducible components of the preimages of the $\frac{1}{5}(1,2)$ singular points. We obtain a surface $Z$ such that $K_Z^2 = 3$ and $\rho(Z)^{G / N} \geqslant 4$. By Corollary \ref{piccrit} there exists a $G / N$-minimal model $Y$ of $Z$ such that $Y$ is a $\ka$-form of a toric surface.

\end{proof}

Now we prove Proposition \ref{DP8}.

\begin{proof}[Proof of Proposition \ref{DP8}]

We can assume that $\XX \cong \Pro^1_{\kka} \times \Pro^1_{\kka}$ and $\rho\left(X\right)^G = 1$ since \mbox{otherwise $\rho\left(X\right)^G = 2$}, so that $X$ admits a \mbox{$G$-equivariant} conic bundle structure by Theorem \ref{Minclass} and $X / G$ is $\ka$-rational by Theorem~\ref{RCbundle}.

Let $f_1: G \rightarrow \Aut\left( \Pro^1_{\kka} \right)$ and $f_2: G \rightarrow \Aut \left( \Pro^1_{\kka} \right)$ be homomorphisms to the groups of automorphisms of the first and the second factor of $\Pro^1_{\kka} \times \Pro^1_{\kka}$ respectively. Then the group
$$
K = \operatorname{Ker} f_1 \times \operatorname{Ker} f_2
$$
\noindent is a normal subgroup of $G$. Then by Lemma \ref{DP8product} the surface $X / K$ is a del Pezzo surface of degree $8$ and
$$
\left( X / K \right) / \left( G / K \right) = X / G.
$$
\noindent So we can replace $X$ by $X / K$ and assume that $K$ is trivial.

Since $K$ is trivial then by Lemma \ref{DP8subgroups} the group $G$ is isomorphic to $A \triangle_A A$ or $(A \triangle_A A) ${\SMALL$\bullet$}$ \CG_2$ where $A$ is one of the following groups: $\CG_n$, $\DG_{2n}$, $\AG_4$, $\SG_4$ or $\AG_5$. For each of these groups we find a normal subgroup $N \lhd G$ such that there exists a $G / N$-MMP-reduction $Y$ of~$X / N$ that is a $\ka$-form of a toric surface.

\begin{itemize}

\item If $G$ is $\CG_2$ then exists an MMP-reduction $Y$ of $X / G$ such that $Y$ is a $\ka$-form of a toric surface by Proposition \ref{cyclictoricquotient}.

\item If $G$ is $\CG_n \triangle_{\CG_n} \CG_n$, $\left(\CG_n \triangle_{\CG_n} \CG_n\right) ${\SMALL$\bullet$}$ \CG_2$, $\DG_{2n} \triangle_{\DG_{2n}} \DG_{2n}$ or $\left(\DG_{2n} \triangle_{\DG_{2n}} \DG_{2n}\right) ${\SMALL$\bullet$}$ \CG_2$ then $N$ is $\CG_n \triangle_{\CG_n} \CG_n$. Any \mbox{$G/N$-MMP-reduction} of $X / N$ is a $\ka$-form of a toric surface by Lemma \ref{DP8cyclic}.

\item If $G$ is $\AG_4 \triangle_{\AG_4} \AG_4$, $\left(\AG_4 \triangle_{\AG_4} \AG_4\right) ${\SMALL$\bullet$}$ \CG_2$, $\SG_4 \triangle_{\SG_4} \SG_4$ or $\left(\SG_4 \triangle_{\SG_4} \SG_4\right) ${\SMALL$\bullet$}$ \CG_2$ then $N$ is $\VG_4 \triangle_{\VG_4} \VG_4$. There exists a $G/N$-MMP-reduction $Y$ of $X / N$ such that $Y$ is a $\ka$-form of a toric surface by Lemma \ref{DP8V4}.

\item If $G$ is $\AG_5 \triangle_{\AG_5} \AG_5$ or $\left(\AG_5 \triangle_{\AG_5} \AG_5\right) ${\SMALL$\bullet$}$ \CG_2$ then $N$ is~$\AG_5 \triangle_{\AG_5} \AG_5$. There exists a $G / N$-MMP-reduction $Y$ of $X / N$ such that $Y$ is a $\ka$-form of a toric surface by Lemmas \ref{DP8A5diag} and~\ref{DP8A5twisted}.

\end{itemize}

Therefore $Y$ is a $\ka$-form of a toric surface. Moreover, $Y(\ka) \ne \varnothing$ since $X(\ka) \ne \varnothing$.

If the surface $Y$ is $\Pro^2_{\ka}$, $\F_n$ or a del Pezzo surface of degree $6$ then
$$
Y / (G/N) \approx X / G
$$
\noindent is $\ka$-rational by Theorems \ref{DP9}, \ref{RCbundle} and \ref{DP6} respectively. If the surface $\overline{Y}$ is $\Pro^1_{\kka} \times \Pro^1_{\kka}$ we apply the procedure above with the smaller group $G/N$. As a result we obtain that $X / G$ is $\ka$-rational.

\end{proof}

\section{Del Pezzo surface of degree $5$}

Let $X$ be a del Pezzo surface of degree $5$. The group $\Aut(\XX)$ is isomorphic to
$$
W(A_4) \cong \SG_5
$$
\noindent (see e.g. \cite[Subsection 6.3]{DI1} or \mbox{\cite[Theorem 8.5.8]{Dol12}}). This group is generated by a subgroup $\SG_4$ and the element $(12345)$. In the notation of Remark \ref{DP_1curves} for any $\sigma \in \SG_4$ one has $\sigma(E_i) = E_{\sigma(i)}$ and \mbox{$\sigma(L_{ij}) = L_{\sigma(i)\sigma(j)}$}.

In this section we prove the following proposition. 

\begin{proposition}
\label{DP5}
Let $X$ be a del Pezzo surface of degree $5$ such that $X(\ka) \ne \varnothing$ and let~$G$ be a subgroup of $\Aut_{\ka}(X)$. Then $X / G$ is $\ka$-rational.
\end{proposition}

To prove Proposition \ref{DP5} we show that in all other cases either the surface $X$ is not $G$-minimal or there is a normal subgroup $N \lhd G$ such that there exists a \mbox{$G / N$-MMP-reduction $Y$} of $X / N$ such that $Y$ is a $\ka$-form of a toric surface. So the proof of Proposition \ref{DP5} is reduced to Theorems \ref{DP9}, \ref{RCbundle}, \ref{DP8}, \ref{DP6}.

\begin{lemma}
\label{DP5nonmin}
Let a finite group $G$ act on a del Pezzo surface $X$ of degree $5$ and $N$ be a nontrivial normal subgroup in $G$. If the group $N$ is isomorphic to $\CG_2$, $\CG_3$ or $\VG_4$ then $X$ is not $G$-minimal.
\end{lemma}

\begin{proof}
If $N \cong \CG_2$ then it is conjugate to $\langle(12)\rangle$ or $\langle(12)(34)\rangle$. In the first case there are exactly four $N$-invariant $(-1)$-curves on $X$: $E_3$, $E_4$, $L_{12}$ and~$L_{34}$. But only the curve $L_{34}$ intersects each other $N$-invariant $(-1)$-curve. Thus $L_{34}$ is $G$-invariant and defined over~$\ka$ so it can be contracted. In the second case there are exactly two orbits consisting of disjoint \mbox{$(-1)$-curves} on $X$: $E_1$ and $E_2$, $E_3$ and $E_4$. Thus this quadruple is $G$-invariant and defined over $\ka$ so it can be contracted.

If $N \cong \CG_3$ then it is conjugate to $\langle(123)\rangle$. There is exactly one $N$-invariant \mbox{$(-1)$-curve $E_4$} on $X$. Thus this curve is $G$-invariant and defined over~$\ka$ so it can be contracted.

If $N \cong \VG_4$ then it is conjugate to $\langle(12)(34), (13)(24)\rangle$. There is exactly one $N$-orbit consisting of four disjoint $(-1)$-curves on $X$: $E_1$, $E_2$, $E_3$ and $E_4$. Thus this quadruple is $G$-invariant and defined over $\ka$ so it can be contracted.
\end{proof}

\begin{lemma}
\label{DP5el5}
Let the group $\CG_5$ act on a del Pezzo surface of degree $5$. Then the group $\CG_5$ has two fixed points on $\XX$ and acts on the tangent spaces of $\XX$ at these points as~$\langle \operatorname{diag}\left( \xi_5, \xi_5^4 \right) \rangle$.
\end{lemma}

\begin{proof}

Let us consider the following Cremona transformation of $\Pro^2_{\kka}$:
$$
g: (x: y: z) \mapsto ( (y - x)z: (z - x)y : yz ).
$$
\noindent The order of $g$ is $5$. This transformation is regular on a del Pezzo surface of degree $5$ which is the blowup of $\Pro^2_{\kka}$ at four points: $(1 : 0 : 0)$, $(0 : 1 : 0)$, $(0 : 0 : 1)$ and $(1 : 1 : 1)$. All del Pezzo surfaces of degree $5$ are isomorphic thus any element of order $5$ is conjugate to $g$ in $\Aut\left( \XX \right)$.

The element $g$ has two fixed points: $\left( \sqrt{5} - 1 : 2 : \sqrt{5} + 1 \right)$ and $\left( \sqrt{5} + 1 : -2 : \sqrt{5} - 1 \right)$ on $\Pro^2_{\kka}$. One can easily check that the element $g$ acts on the tangent spaces of $\Pro^2_{\kka}$ at the fixed points as $\operatorname{diag}\left( \xi_5, \xi_5^4 \right)$.

\end{proof}

\begin{lemma}
\label{DP5C5}
Let a finite group $G$ act on a del Pezzo surface $X$ of degree $5$ and $N \cong \CG_5$ be a normal subgroup in $G$. Then there exists a \mbox{$G / N$-MMP-reduction $Y$} of $X / N$ such that $\overline{Y} \cong \Pro^2_{\kka}$.
\end{lemma}

\begin{proof}
By Lemma \ref{DP5el5} on the quotient $\XX / N$ there are two~$A_4$ singularities, $-K_{\XX / N}$ is ample and
$$
K_{X_N}^2 = \frac{K_X^2}{5} = 1.
$$
\noindent It means that $X / N$ is a singular del Pezzo surface of degree $1$.

By Lemma \ref{DP1sing2A4} there exists a $G / N$-MMP-reduction $Y$ of $X / N$ such that $\overline{Y} \cong \Pro^2_{\kka}$.
\end{proof}

\begin{lemma}
\label{DP5A5}
Let a finite group $G$ act on a del Pezzo surface $X$ of degree $5$ and $N \cong \AG_5$ be a normal subgroup in $G$. Then there exists a \mbox{$G / N$-MMP-reduction $Y$} of $X / N$ that is isomorphic to $\F_3$.
\end{lemma}

\begin{proof}
Let us consider fixed points of elements of $N$. The stabilizer of such a point is a subgroup of $\AG_5$ having a faithful representation in~$\mathrm{GL}_2(\kka)$ (see e.g. \cite[Lemma 4]{Pop14}). Any subgroup of $\AG_5$ is isomorphic to $\CG_2$, $\CG_3$, $\VG_4$, $\SG_3$, $\AG_4$, $\CG_5$, $\DG_{10}$ or $\AG_5$. The groups~$\AG_4$ and $\AG_5$ do not have faithful two-dimensional representations. For the groups $\VG_4$, $\SG_3$ and~$\DG_{10}$ faithful two-dimensional representations are generated by reflections thus images of points with such stabilizers are smooth points on the quotient surface. All other groups are cyclic groups of prime order.

An element of order $5$ has exactly two fixed points on $X$ and this element acts on the tangent spaces of $\XX$ at the fixed points as $\operatorname{diag}\left( \xi_5, \xi_5^4 \right)$ by Lemma \ref{DP5el5}.

Each element of order $3$ in $\AG_5$ is conjugate to $\langle (123) \rangle$. The unique $(-1)$-curve on $X$ invariant with respect to the group $\langle (123) \rangle$ is $E_4$. Let us $\langle (123) \rangle$-equivariantly contract the four $(-1)$-curves $E_i$ and get $\Pro^2_{\kka}$. The group $\langle(123)\rangle$ acts on $\Pro^2_{\kka}$ and has no curves of fixed points since the line passing through $p_1$ and $p_2$ (see the notation of Remark \ref{DP_1curves}) does not contain $\langle (123) \rangle$-fixed points. Therefore the action of $\langle(123)\rangle$ on $\Pro^2_{\kka}$ is conjugate to $\langle\mathrm{diag}(1, \omega, \omega^2)\rangle$ and it has $3$ fixed points one of which is~$p_4$ by Lemma~\ref{C3fixedpoint}. In the tangent space of $\XX$ at the other two fixed points the group $\langle(123)\rangle$ acts as $\langle\mathrm{diag}(\omega, \omega^2)\rangle$. On the $(-1)$-curve $E_4$ the group $\langle(123)\rangle$ has two fixed points. This group acts on the tangent space of $\XX$ at these points as $\langle\mathrm{diag}(\omega, \omega)\rangle$ by Remark \ref{C3fixed}. There are $20$ elements of order $3$ in $\AG_5$ and they have $20$ fixed points on $(-1)$-curves. The stabilizer of such a point is $\CG_3$ so all these points are permuted by the group $\AG_5$.

Consider a group
$$
\VG_4 = \langle(12)(34), (13)(24)\rangle.
$$
\noindent One can $\VG_4$-equivariantly contract the four $(-1)$-curves $E_i$ and get~$\Pro^2_{\kka}$. The group $\VG_4$ acts on $\Pro^2_{\kka}$ and each nontrivial element has a line of fixed points not passing through the points~$p_i$. Thus each of these elements has a curve of fixed points in $\XX$ whose class in $\Pic(\XX)$ is~$L$.

The images of $L$ in $\Pic(\XX)$ under the action of $\CG_5$ are
$$
2L - E_1 - E_2 - E_3, 2L - E_1 - E_2 - E_4, 2L - E_1 - E_3 - E_4\,\,\textrm{and}\,\,2L - E_2 - E_3 - E_4.
$$
\noindent Thus the ramification divisor of the quotient morphism $f: X \rightarrow X / \AG_5$ is a member of the linear system $|-9K_X|$. By the Hurwitz formula
$$
K_{X / \AG_5}^2 = \frac{1}{60}(K_X + 9K_X)^2 = \frac{25}{3}.
$$
\noindent Moreover, there is one $\frac{1}{3}(1, 1)$ singularity and maybe some Du Val singularities on $\XX / \AG_5$. Let $\pi: Y \rightarrow X / \AG_5$ be the minimal resolution of singularities. One has
$$
K_Y^2 = K_{X / \AG_5}^2 - \frac{1}{3} = 8, \quad \rho(\overline{Y}) = 10 - K_{\overline{Y}}^2 = 2.
$$ 
\noindent Therefore the only singularity on $\XX / \AG_5$ is $\frac{1}{3}(1, 1)$ and $Y$ is isomorphic to $\F_3$ since $K_Y^2 = 8$ and $Y$ contains a $(-3)$-curve.
\end{proof}

Now we prove Proposition \ref{DP5}.

\begin{proof}[Proof of Proposition \ref{DP5}]

By Lemma \ref{S5subgroups} each group $G \subset \SG_5$ has a normal subgroup $N$ isomorphic to $\CG_2$, $\CG_3$, $\VG_4$, $\CG_5$ or $\AG_5$.

If $N$ is isomorphic to $\CG_2$, $\CG_3$ or $\VG_4$ then $X$ is not $G$-minimal by Lemma \ref{DP5nonmin} and $X / G$ is $\ka$-rational by Theorems \ref{DP9}, \ref{RCbundle}, \ref{DP8} and~\ref{DP6}.

If $N$ is isomorphic to $\CG_5$ then there exists a $G / N$-MMP-reduction $Y$ of $X / N$ such that~$\overline{Y}$ is isomorphic to $\Pro^2_{\kka}$ by Lemma \ref{DP5C5}.

If $N$ is isomorphic to $\AG_5$ then any $G / N$-MMP-reduction $Y$ of $X / N$ is isomorphic to $\F_3$ by Lemma \ref{DP5A5}.

In the last two cases one has $Y(\ka) \ne \varnothing$ since $X(\ka) \ne \varnothing$. Thus
$$
Y / (G / N) \approx X / G
$$
\noindent is $\ka$-rational by Theorem \ref{DP9} and \ref{RCbundle} respectively.
\end{proof}

Now we can prove Corollary \ref{geq5}.

\begin{proof}[Proof of Corollary \ref{geq5}]

Let $Y$ be a $G$-minimal model of $X$. Then $K_Y^2 \geq K_X^2 \geq 5$ by Theorem \ref{GMMP}, $Y(\ka) \ne \varnothing$ since $X(\ka) \ne \varnothing$ and $Y$ is either a del Pezzo surface or admit a conic bundle structure by Theorem \ref{Minclass}.

Therefore $X / G \approx Y / G$ is $\ka$-rational by Theorems \ref{DP9}, \ref{RCbundle}, \ref{DP8}, \ref{DP6} and \ref{DP5}.

\end{proof}

\section{Del Pezzo surface of degree $4$}

Let $X$ be a del Pezzo surface of degree $4$. The group $\Aut(\XX)$ is a subgroup of the group
$$
W(D_5) \cong \CG_2^4 \rtimes \SG_5
$$
\noindent (see e.g. \cite[Subsection 6.4]{DI1} or \cite[Proposition 8.6.7]{Dol12}). The group $\CG_2^4 \rtimes \SG_5$ is generated by subgroups $\SG_5$ and $\CG_2^4$. In the notation of Remark \ref{DP_1curves} for any $\sigma \in \SG_5$ one has $\sigma(E_i) = E_{\sigma(i)}$, $\sigma(L_{ij}) = L_{\sigma(i)\sigma(j)}$ and~$\sigma(Q) = Q$.

The surface $\XX$ is isomorphic to a surface of degree $4$ in $\Pro^4_{\kka}$ given by equations
$$
\sum \limits_{i=1}^5 x_i^2 = 0, \qquad \sum \limits_{i=1}^5 a_i x_i^2 = 0.
$$
\noindent The group $\CG_2^4$ acts on $\Pro^4_{\kka}$ and $\XX$ as a diagonal subgroup of $\mathrm{PGL}_5\left(\kka\right)$. There are involutions of two kinds in such diagonal group: $\iota_{ijkl}$ and~$\iota_{ij}$. These involutions switch signs of coordinates $x_i$, $x_j$, $x_k$, $x_l$ and $x_i$, $x_j$, respectively.

In this section we prove the following proposition.

\begin{proposition}
\label{DP4}
Let $X$ be a del Pezzo surface of degree $4$ such that $X(\ka) \ne \varnothing$ and let~$G$ be a subgroup of $\Aut_{\ka}(X)$. Then $X / G$ is $\ka$-rational if $G$ is not conjugate to any of the groups $\langle id \rangle$, $\CG_2 = \langle\iota_{12}\rangle$, $\VG_4 = \langle\iota_{12}, \iota_{13}\rangle$ or $\CG_4 = \langle(12)(34)\iota_{15}\rangle$.
\end{proposition}

In Section 6 we will show that in the latter three cases the quotient can be non-$\ka$-rational. Now we show that in all other cases the quotient of $X$ is $\ka$-rational.

To prove Proposition \ref{DP4} we show that in any of the remaining cases either the surface $X$ is not $G$-minimal or there is a normal subgroup $N \lhd G$ such that there exists a \mbox{$G / N$-MMP-reduction $Y$} of $X / N$ such that $Y$ is a $\ka$-form of a toric surface. So the proof of Proposition \ref{DP4} is reduced to Theorems \ref{DP9}, \ref{RCbundle}, \ref{DP8}, \ref{DP6} and \ref{DP5}.

Now we are going to prove some auxillary lemmas.

The following lemma immediately follows from the results of \cite[Subsection 6.4]{DI1}. We give the proof for convenience of the reader. 

\begin{lemma}
\label{DP4nonact}
Let a finite group $G$ act on a del Pezzo surface $X$ of degree $4$ and
$$
h:\Aut(X) \rightarrow \SG_5
$$
\noindent be the natural homomorphism. Then the group $h(G)$ does not contain subgroups conjugate to $\CG_2 = \langle(12)\rangle$ \mbox{and $\VG_4 = \langle(12)(34), (13)(24)\rangle$.}
\end{lemma}

\begin{proof}
The group $\CG_2^4$ acts on $\XX$. The group $\Aut(\XX) \subset \CG_2^4 \rtimes \SG_5$ contains a subgroup~$\CG_2^4$. Therefore if the group $G$ contains an element $hg$, where $g \in \SG_5$, $h \in\CG_2^4$ then the \mbox{group $\Aut(\XX)$} contains the element $g$. Thus it is sufficient to prove that there are no subgroups in $\Aut(\XX)$ conjugate to $\CG_2 = \langle(12)\rangle$ \mbox{and $\VG_4 = \langle(12)(34), (13)(24)\rangle$}.

Suppose that the group $\CG_2 = \langle(12)\rangle$ acts on $\XX$. One can $\CG_2$-equivariantly contract five $(-1)$-curves $E_1$, $E_2$, $E_3$, $E_4$ and $E_5$ and get $\Pro^2_{\kka}$ with the action of~$\CG_2$. The group $\CG_2$ has a unique isolated fixed point on $\Pro^2_{\kka}$ and a unique line of fixed points on~$\Pro^2_{\kka}$. The points $p_3$, $p_4$ and $p_5$ (see the notation of Remark \ref{DP_1curves}) are fixed by the group~$\CG_2$. These three points do not lie on a line so one of these points is the isolated fixed point. The points $p_1$ and $p_2$ are permuted by the group $\CG_2$. Therefore the group $\CG_2$ acts faithfully on the line passing through these two points. Thus this line contains the isolated fixed point of $\CG_2$. The proper transform of this line on $\XX$ is a $(-2)$-curve but there are no $(-2)$-curves on a del Pezzo surface of degree $4$. Therefore the group $\CG_2 = \langle (12) \rangle$ cannot act on $\XX$.

Suppose that the group $\VG_4 = \langle(12)(34), (13)(24)\rangle$ acts on $\XX$. One can \mbox{$\VG_4$-equivariantly} contract five $(-1)$-curves $E_1$, $E_2$, $E_3$, $E_4$ and $E_5$ and get $\Pro^2_{\kka}$ with the action of $\VG_4$. The point $p_5$ is fixed by the group $\VG_4$. Thus this point is the unique isolated fixed point on $\Pro^2_{\kka}$ of an element of~$\VG_4$. Therefore as above in the case of the group $\CG_2 = \langle (12) \rangle$ three of the points $p_1$, $p_2$, $p_3$, $p_4$ and $p_5$ lie on a line and the group $\VG_4$ cannot act on $\XX$.
\end{proof}

\begin{lemma}
\label{DP4C2}
Let a finite group $G$ act on a del Pezzo surface $X$ of degree~$4$ and 
$$
N \cong \CG_2 = \langle(12)(34)\rangle
$$
\noindent be a normal subgroup in $G$. Then there exists a \mbox{$G / N$-MMP-reduction $Y$} of $X / N$ such that $Y$ is a $\ka$-form of a toric surface.
\end{lemma}

\begin{proof}
Let us $N$-equivariantly contract five $(-1)$-curves $E_1$, $E_2$, $E_3$, $E_4$ and $E_5$ on $\XX$ and get a $\Pro^2_{\kka}$ with the action of $N$. The group $N$ has a unique isolated fixed point on $\Pro^2_{\kka}$ and a unique line of fixed points. As in the proof of Lemma~\ref{DP4nonact} the point $p_5$ lies on this line. Thus on the surface $\XX$ the group~$N$ has $2$ isolated fixed points $L_{12} \cap L_{34}$, $Q \cap E_5$ and a curve of fixed points whose class in $\Pic(\XX)$ is $L - E_5$.

Let $f: X \rightarrow X / N$ be the quotient morphism and
$$
\pi: \widetilde{X / N} \rightarrow X / N
$$
\noindent be the minimal resolution of singularities. By the Hurwitz formula
$$
K_{X/N}^2 = \frac{1}{2}\left(K_X - L + E_5 \right)^2 = 4.
$$
\noindent There are exactly two $A_1$ singularities on $X / N$. The proper transforms $\pi^{-1}_*f\left( L_{12} \right)$, $\pi^{-1}_*f\left( L_{34} \right)$, $\pi^{-1}_*f\left( Q \right)$ and $\pi^{-1}_*f\left( E_5 \right)$ are four disjoint $G / N$-invariant $(-1)$-curves defined over $\ka$ (see Table \ref{table1}). One can $G / N$-equivariantly contract this quadruple and get a surface $Y$ such that $K_Y^2 = 8$. Thus there exists a $G/N$-MMP-reduction $Y$ of $X / N$ such that $Y$ is a $\ka$-form of a toric surface by Lemma \ref{torcrit}.
\end{proof}

\begin{remark}
\label{maincond}

Note that any involution in $W(D_5)$ is conjugate to $(12)$, $(12)(34)$, $\iota_{12}$ or~$\iota_{1234}$. By Lemma \ref{DP4nonact} an element conjugate to $(12)$ cannot act on a del Pezzo surface of degree~$4$. From the proof of Lemma \ref{DP4C2} one can see that an element conjugate to $(12)(34)$ has a curve of fixed points, and for an element conjugate to $\iota_{1234}$ there exists a hyperplane section consisting of fixed points. Thus if an involution acting on a del Pezzo surface of degree $4$ does not have curves of fixed points then it is conjugate to $\iota_{12}$. For the groups of order $4$ there are only two cases, for which all elements of order $2$ are conjugate to $\iota_{12}$: either $\langle \iota_{12}, \iota_{13} \rangle$, or $\langle (12)(34)\iota_{15} \rangle$. Therefore the conditions on the group $G$ in Proposition~\ref{DP4} are that $\ord G$ is $1$, $2$ or $4$, and nontrivial elements of $G$ do not curves of fixed points.

\end{remark}

\begin{lemma}
\label{DP4C3}
Let a finite group $G$ act on a del Pezzo surface $X$ of degree $4$ and
$$
N \cong \CG_3 = \langle(123)\rangle
$$
\noindent be a normal subgroup in $G$. Then there exists a $G / N$-MMP-reduction $Y$ of $X / N$ such that $Y$ is a $\ka$-form of a toric surface.
\end{lemma}

\begin{proof}
Let $\sigma: \XX \rightarrow \Pro^2_{\kka}$ be the $N$-equivariant contraction of the five $(-1)$-curves $E_1$, $E_2$, $E_3$, $E_4$ and $E_5$. The group~$N$ has three isolated fixed points on $\Pro^2_{\kka}$ two of which are $p_4$ and~$p_5$ (see the notation of Remark \ref{DP_1curves}). Denote the third fixed point by $p$. The group~$N$ acts on the tangent space of $\Pro^2_{\kka}$ at these fixed points as $\langle\operatorname{diag}\left( \omega, \omega^2 \right)\rangle$ by Lemma~\ref{C3fixedpoint}. Therefore there are five fixed points on $\XX$: $E_4 \cap L_{45}$, $E_5 \cap L_{45}$, $E_5 \cap Q$, $E_4 \cap Q$ and $\sigma^{-1}(p)$. The group $N$ acts on the tangent space of $\XX$ at the point $\sigma^{-1}(p)$ as $\langle\operatorname{diag}\left( \omega, \omega^2 \right)\rangle$ and on the tangent spaces of $\XX$ at the other fixed points as $\langle\operatorname{diag}\left( \omega, \omega \right)\rangle$ by Remark \ref{C3fixed}.

Let $C_1$, $C_2$, $C_3$ and $C_4$ be $N$-invariant curves on $X$ with classes
$$
2L - E_1 - E_2 - E_3 - E_4, \quad 2L - E_1 - E_2 - E_3 - E_5, \quad L - E_5 \quad \textrm{and} \quad L - E_4
$$
\noindent passing through $\sigma^{-1}p$ and another fixed point (these curves are proper transforms of lines passing through $p$ and $p_4$ or $p$ and $p_5$, and conics passing through $p$, $p_1$, $p_2$, $p_3$, $p_4$ \mbox{or $p$, $p_1$, $p_2$, $p_3$, $p_5$}).

Assume that there is another $N$-invariant irreducible rational curve $C$ with self-intersection number $0$ on $X$ whose class in $\Pic \left( \XX \right)$ is
$$
aL - b \left( E_1 + E_2 + E_3 \right) - cE_4 - dE_5.
$$
\noindent The curve $C$ is irreducible thus the numbers $a$, $b$, $c$, $d$ are non-negative. One has $C^2 = 0$ \mbox{and $C(K_X + C) = -2$}. That means that the following system of equations holds:
$$
\begin{cases}
a^2 - 3b^2 - c^2 - d^2 = 0, \\
3a - 3b - c - d = 2.
\end{cases}
$$
\noindent One can check that all possibilities for the class of $C$ are
$$
2L - E_1 - E_2 - E_3 - E_4, \quad 2L - E_1 - E_2 - E_3 - E_5, \quad L - E_5 \quad \textrm{and} \quad L - E_4.
$$
\noindent Therefore there are no rational curves on $X$ with self-intersection number $0$ which differ \mbox{from $C_1$, $C_2$, $C_3$ and $C_4$}.

Let $f: X \rightarrow X / N$ be the quotient morphism and
$$
\pi: \widetilde{X / N} \rightarrow X / N
$$
\noindent be the minimal resolution of singularities. Then there are four $\frac{1}{3}(1,1)$ singularities \mbox{$f(E_4 \cap L_{45})$}, $f(E_5 \cap L_{45})$, $f(E_5 \cap Q)$, $f(E_4 \cap Q)$ and one $A_2$ singularity $f \left(\sigma^{-1}(p)\right)$ on~$\XX / N$. Consider $8$ curves $f(C_1)$, $f(C_2)$, $f(C_3)$, $f(C_4)$, $f(E_4)$, $f(L_{45})$, $f(E_5)$ and $f(Q)$. This eighttuple is $G / N$-equivariant and defined over $\ka$. The intersection number of any two of these curves is $0$, $\frac{1}{3}$ or $\frac{2}{3}$. Thus their proper transforms on $\widetilde{\XX / N}$ are eight disjoint \mbox{$(-1)$-curves} (see Table \ref{table1}). One can $G / N$-equivariantly contract these curves and get a surface $Y$. Then
$$
K_Y^2 = K_{\widetilde{X / N}}^2 + 8 = K_{X / N}^2 - \frac{4}{3} + 8 = \frac{1}{3} K_X^2 + \frac{28}{3} = 8.
$$
Thus there exists a $G/N$-MMP-reduction $Y$ of $X / N$ such that $Y$ is a $\ka$-form of a toric surface by Lemma \ref{torcrit}.
\end{proof}

\begin{lemma}
\label{DP4C5}
Let a finite group $G$ act on a del Pezzo surface $X$ of degree $4$ and
$$
N \cong \CG_5 = \langle(12345)\rangle
$$
\noindent be a normal subgroup in $G$. Then $X$ is not $G$-minimal.
\end{lemma}

\begin{proof}
The group $N$ has a unique invariant $(-1)$-curve $Q$ on $X$. Thus this curve is \mbox{$G$-invariant} and defined over $\ka$ so it can be contracted.
\end{proof}

\begin{lemma}
\label{DP4i1234}
Let a finite group $G$ act on a del Pezzo surface $X$ of degree $4$ and contain an element conjugate to $\iota_{1234}$. Then there exists a normal subgroup $N \lhd G$ such that a $G / N$-MMP-reduction $Y$ of $X / N$ is a $\ka$-form of a toric surface.
\end{lemma}

\begin{proof}
Note that the set of fixed points of an element conjugate to $\iota_{1234}$ is a hyperplane section of $X$ in $\Pro^4_{\ka}$.

Let $N$ be a subgroup of $G$ generated by elements conjugate to $\iota_{1234}$. Note that each element conjugate to $\iota_{1234}$ has a curve of fixed points on $X$ that is a member of the linear \mbox{system $|-K_X|$}. The group $N$ is normal and one of the following possibilities holds:

\begin{itemize}

\item If $N$ is generated by one element conjugate to $\iota_{1234}$ then $N \cong \CG_2$ and by the Hurwitz formula
$$
K_{X / N}^2 = \frac{1}{2} \left( 2K_X \right)^2 = 8.
$$

\item If $N$ is generated by two elements conjugate to $\iota_{1234}$ then $N \cong \CG_2^2$ and by the Hurwitz formula
$$
K_{X / N}^2 = \frac{1}{4} \left( 3K_X \right)^2 = 9.
$$

\item If $N$ is generated by three elements conjugate to $\iota_{1234}$ then $N \cong \CG_2^3$ and by the Hurwitz formula
$$
K_{X / N}^2 = \frac{1}{8} \left( 4K_X \right)^2 = 8.
$$

\item If $N$ is generated by four elements conjugate to $\iota_{1234}$ then $N$ contains the fifth element conjugate to $\iota_{1234}$. One has $N \cong \CG_2^4$ and by the Hurwitz formula
$$
K_{X / N}^2 = \frac{1}{16} \left( 6K_X \right)^2 = 9.
$$

\end{itemize}

The surface $X / N$ has at worst Du Val singularities. Hence for any \mbox{$G / N$-MMP-reduction~$Y$} of $X / N$ one has $K_Y^2 \geqslant K_{X / N}^2 \geqslant 8$. Thus $Y$ is a $\ka$-form of a toric surface by Lemma~\ref{torcrit}.

\end{proof}

\begin{lemma}
\label{DP4i12i13}
Let a finite group $G$ act on a del Pezzo surface $X$ of degree $4$ and
$$
N \cong \VG_4 = \langle \iota_{12}, \iota_{13} \rangle
$$
\noindent be a normal subgroup in $G$. Then the surface $X / N$ is $G / N$-birationally equivalent to a del Pezzo surface $Y$ of degree $4$. If $\rho(X)^G = 1$ and for each nontrivial element of $N$ all its fixed points are in one orbit of the group $G \times \Gal \left( \kka / \ka \right)$ then $Y$ is $G / N$-minimal.
\end{lemma}

\begin{proof}

Each non-trivial element of~$N$ has $4$ fixed points on $X$. The hyperplane sections $x_1 = 0$, $x_2 = 0$ and $x_3 = 0$ cut out from $X$ elliptic curves $C_1$, $C_2$ and $C_3$ defined over~$\ka$. Each of these curves contains $8$ points each of which is fixed by a non-trivial element of~$N$.

Let $f: X \rightarrow X / N$ be the quotient map and
$$
\pi: \widetilde{X / N} \rightarrow X / N
$$
\noindent be the minimal resolution of singularities. Then $f \left( C_i \right)$ is a rational curve containing four $A_1$ singularities. One has
$$
f \left( C_i \right) \cdot f \left( C_j \right) = \frac{1}{4} C_i \cdot C_j = \frac{1}{4} K_X^2 = 1.
$$
Thus $\pi^{-1}_*f \left( C_1 \right)$, $\pi^{-1}_*f \left( C_2 \right)$ and $\pi^{-1}_*f \left( C_3 \right)$ are three disjoint $(-1)$-curves defined over $\ka$ (see Table \ref{table1}). We can contract these three curves and get a surface $Y$ such that
$$
K_Y^2 = K_{\widetilde{X / N}}^2 + 3 = K_{X / N}^2 + 3 = \frac{1}{4} K_X^2 + 3 = 4.
$$
\noindent Note that the curves $\pi^{-1}_*f \left( C_1 \right)$, $\pi^{-1}_*f \left( C_2 \right)$ and $\pi^{-1}_*f \left( C_3 \right)$ intersect all $(-2)$-curves on~$\widetilde{X / N}$, since the curves $f \left( C_1 \right)$, $f \left( C_2 \right)$ and $f \left( C_3 \right)$ pass through all $A_1$ singular points on $X / N$. Therefore $Y$ does not contain curves with self-intersection less than $-1$. Thus $Y$ is a del Pezzo surface of degree $4$.

Suppose that $\rho(X)^G = 1$ and for each nontrivial element of $N$ all its fixed points are in one orbit of the group $G \times \Gal \left( \kka / \ka \right)$. Let $k$ be the number of conjugacy classes in $G$ containing non-trivial elements of $N$. Then the fixed points of non-trivial elements of $N$ lie in $k$ orbits of the group $G \times \Gal \left( \kka / \ka \right)$ and curves $C_1$, $C_2$ and $C_3$ form $k$ orbits of the group $G \times \Gal \left( \kka / \ka \right)$ (so that in particular $k \leqslant 3$). Therefore
$$
\rho(Y)^{G / N} = \rho \left( \widetilde{ X / N } \right)^{G / N} - k = \rho \left( X / N \right)^{G / N} = \rho(X)^G = 1,
$$
\noindent Thus $Y$ is $G / N$-minimal.

\end{proof}

Now we prove Proposition \ref{DP4}.

\begin{proof}[Proof of Proposition \ref{DP4}]

If the group $G$ contains an element conjugate to $\iota_{1234}$ then by Lemma \ref{DP4i1234} there exists a normal subgroup $N \lhd G$ such that a $G / N$-MMP-reduction $Y$ of~$X / N$ is a $\ka$-form of a toric surface.

If the group $G \cap \CG_2^4$ is conjugate to the group $\VG_4 = \langle\iota_{12}, \iota_{13}\rangle$ then the group $G$ is conjugate to a subgroup of
$$
\VG_4 \rtimes \left( \SG_3 \times \CG_2 \right) = \langle\iota_{12}, \iota_{13}, (123), (12), (45)\rangle.
$$
\noindent Such a group $G$ cannot contain a subgroup conjugate to $\CG_2 = \langle(12)\rangle$ by Lemma \ref{DP4nonact}. If the group $G$ does not contain an element of order $3$ then either $G = \VG_4$ or $G$ contains a normal subgroup conjugate to $N = \langle(12)(45)\rangle$ and there exists a $G / N$-MMP-reduction $Y$ of $X / N$ such that $Y$ is a $\ka$-form of a toric surface by Lemma \ref{DP4C2}. Otherwise by Lemma~\ref{DP4i12i13} the quotient $X / N$ is $G / N$-birationally equivalent to a del Pezzo surface $Z$ of degree~$4$ and the group $G / N$ contains an element of order $3$. So we can replace $X$ by $Z$, $G$ by~$G / N$ and start the proof from the beginning with a smaller group.

If the group $G \cap \CG_2^4$ is conjugate to the group $\CG_2 = \langle\iota_{12}\rangle$ then the group~$G$ is conjugate to a subgroup of
$$
\CG_2 \times \left(\CG_2 \times \SG_3 \right) = \langle\iota_{12}, (12), (345), (34)\rangle.
$$
\noindent Such a group $G$ cannot contain a subgroup conjugate to $\CG_2 = \langle(12)\rangle$ by Lemma \ref{DP4nonact}. If such a group $G$ contains a subgroup conjugate to $N = \langle(345)\rangle$ then this group is normal and there exists a \mbox{$G/N$-MMP-reduction $Y$} of $X / N$ such that $Y$ is a $\ka$-form of a toric surface by Lemma \ref{DP4C3}. Otherwise either the group $G$ is conjugate to $\langle\iota_{12}\rangle$, or $\langle\iota_{15}(12)(34)\rangle$ or the group $G$ contains a normal subgroup conjugate to $N = \langle(12)(34)\rangle$ and there exists a $G / N$-MMP-reduction $Y$ of $X / N$ such that $Y$ is a $\ka$-form of a toric surface by Lemma \ref{DP4C2}.

If the group $G \cap \CG_2^4$ is trivial then $G$ is isomorphic to a subgroup of~$\SG_5$. By Lemma~\ref{DP4nonact} the group $G$ cannot contain subgroups conjugate to~$\CG_2 = \langle(12)\rangle$, $\VG_4$ and $\AG_5$. Thus by Lemma \ref{S5subgroups} the group $G$ contains a normal subgroup $N$ conjugate to $\CG_2 = \langle(12)(34)\rangle$, $\CG_3 = \langle(123)\rangle$ or~$\CG_5 = \langle(12345)\rangle$. In the last case the surface $X$ is not $G$-minimal by Lemma \ref{DP4C5} and the quotient $X / G$ is $\ka$-rational by Theorems \ref{DP9}, \ref{RCbundle}, \ref{DP8}, \ref{DP6} and \ref{DP5}. In the other two cases there exists a $G / N$-MMP-reduction $Y$ of $X / N$ such that $Y$ is a $\ka$-form of a toric surface by Lemmas \ref{DP4C2} and \ref{DP4C3}.

In all cases $Y(\ka) \ne \varnothing$ since $X(\ka) \ne \varnothing$. Thus
$$
Y / (G / N) \approx X / G
$$
\noindent is $\ka$-rational by Theorems \ref{DP9}, \ref{RCbundle}, \ref{DP8} and \ref{DP6}.
\end{proof}

\section{Examples of nonrational quotients}

In this section we show that for the groups $\CG_2 = \langle\iota_{12}\rangle$, $\VG_4 = \langle\iota_{12}, \iota_{13}\rangle$ \mbox{and $\CG_4 = \langle(12)(34)\iota_{15}\rangle$} acting on a del Pezzo surface of degree $4$ the quotient may be non-$\ka$-rational. We use the notation of Section 5.

We start with the quotients of a del Pezzo surface of degree $4$ by the group $\CG_2 = \langle\iota_{12}\rangle$.

\begin{lemma}
\label{DP4i12}
Let a finite group $G$ act on a del Pezzo surface $X$ of degree $4$ and
$$
N \cong \CG_2 = \langle \iota_{12} \rangle
$$
\noindent be a normal subgroup in $G$. Then the surface $X / N$ is $G / N$-birationally equivalent to a conic bundle $Y$ with $K_Y^2 = 2$. If $\rho(X)^G = 1$ and all fixed points of $\iota_{12}$ are in one orbit of the group $G \times \Gal \left( \kka / \ka \right)$ then $Y$ is $G / N$-minimal.
\end{lemma}

\begin{proof}

The element $\iota_{12}$ has four fixed points on $X$ cut out by the plane $x_1 = x_2 = 0$. Let $\mathcal{C}$ be a one-dimensional linear subsystem of $|-K_X|$ spanned by the curves $x_1 = 0$ and~$x_2 = 0$. A general member of $\mathcal{C}$ is an $N$-invariant elliptic curve passing through all fixed points of $N$.

Let $f: X \rightarrow X / N$ be the quotient morphism and $\pi: Y \rightarrow X / N$ be the minimal resolution of singularities. A general member of the linear system $f_* \mathcal{C}$ is a smooth conic passing through four $A_1$ singularities of $X / N$. Therefore the linear system $\pi^{-1}_* f_* \mathcal{C}$ gives a conic bundle structure
$$
\varphi_{\pi^{-1}_* f_* \mathcal{C}}: Y \rightarrow \Pro^1_{\ka}.
$$
\noindent One has
$$
K_Y^2 = K_{X / N}^2 = \frac{1}{2} K_X^2 = 2.
$$

If all fixed points of $\iota_{12}$ are in one orbit of the group $G \times \Gal \left( \kka / \ka \right)$ then
$$
\rho(Y)^{G / N} = \rho( X / N )^{G / N} + 1 = \rho(X)^G + 1.
$$
Thus if $\rho(X)^G = 1$ then $\rho(Y)^{G / N} = 2$ and $Y$ is $G / N$-minimal by Theorem \ref{MinCB}(iii).

\end{proof}

\begin{remark}
\label{DP4i12rat}
Note that in Lemma \ref{DP4i12} if one of the $\iota_{12}$-fixed points $p \in \XX$ is defined over~$\ka$ and $G$-fixed, then $\pi^{-1} \left( f (p) \right)$ is a section $E$ of the conic bundle $Y \rightarrow \Pro^1_{\ka}$ defined over $\ka$ such that~$E^2 = -2$. The conic bundle $Y \rightarrow \Pro^1_{\ka}$ has $6$ singular fibres since $K_Y^2 = 2$. Therefore one can $G / N$-equivariantly contract $6$ components of singular fibres meeting $E$ and get a surface of degree~$8$ which is $\ka$-rational by Theorem \ref{ratcrit}.
\end{remark}

Now we construct an explicit example satysfying the conditions of Lemma \ref{DP4i12}.

\begin{example}
\label{DP4i12exConst}
Consider a surface $X$ in $\Pro^4_{\ka}$ given by the equations

\begin{equation}
\label{DP4example}
x_1^2 + x_3^2 - x_4^2 -x_5^2 = 0, \qquad -x_2^2+2x_3^2-x_4^2-4x_5^2=0.
\end{equation}
\noindent Note that $X(\ka) \ne \varnothing$ since the $\ka$-point $(0:1:1:1:0)$ lies on $X$.

The normal subgroup $\CG_2^4$ of $W(D_5)$ acts on $X$ by switching signs of coordinates: elements $\iota_{ij}$ and $\iota_{ijkl}$ switch signs of coordinates $x_i$, $x_j$ and $x_i$, $x_j$, $x_k$, $x_l$, respectively. In particular, the group $G = \langle \iota_{12} \rangle$ acts on $\Pro^4_{\ka}$ and switches signs of coordinates $x_1$ and~$x_2$. The group $G$ has four fixed points $\left(0:0:\pm\sqrt{3}:\pm\sqrt{2}:1\right)$ on $\XX$.

One can check that sixteen $(-1)$-curves on the surface $\XX$ are given by the following parametrization:
$$
\left( \pm \left( \alpha x + y \right) : \pm \left( i \left( 1 + \sqrt{6} \right)x + i\left(2\sqrt{2} + \sqrt{3}\right)y \right) : \pm \left( x - \alpha y \right) : \pm \left(\alpha x - y \right) : \pm \left( x + \alpha y \right) \right)
$$
\noindent where $\alpha = \sqrt{2} + \sqrt{3}$. These curves are defined over any field containing $i$, $\sqrt{2}$ and $\sqrt{3}$.

\end{example}

\begin{example}
\label{DP4i12ex}

Suppose that in Example \ref{DP4i12exConst} the field $\ka$ does not contain $i$, $\sqrt{2}$ and~$\sqrt{3}$, and
$$
\Gal \left( \ka \left( i, \sqrt{2}, \sqrt{3} \right) / \ka \right) \cong \CG_2^3.
$$
\noindent For instance, this holds for $\ka = \mathbb{Q}$. In this case the image of the group $\Gal \left( \ka \left( i, \sqrt{2}, \sqrt{3} \right) / \ka \right)$ in $W(D_5)$ is $\langle \iota_{1345}, \iota_{15}, \iota_{45} \rangle$. The surface $X$ admits a structure of a minimal conic bundle by Theorem \ref{MinCB}(iii), since $\rho(X) = 2$. Thus the surface $X$ is not $\ka$-rational by Theorem \ref{ratcrit}. One has $\rho(X)^G = 1$ and the four fixed points of $G$ on $X$ are permuted by the Galois group. Thus the surface $Y$ admits a structure of a minimal conic bundle with $K_Y^2 = 2$ by Lemma \ref{DP4i12} and $Y$ is not $\ka$-rational by Theorem \ref{ratcrit}. This gives us an example of a non-$\ka$-rational quotient of a ($G$-minimal) non-$\ka$-rational del Pezzo surface of degree $4$ by the group $\CG_2$.

Now assume that in Example \ref{DP4i12exConst} the field $\ka$ contains $i\sqrt{2}$ but does not contain $\sqrt{2}$, $\sqrt{3}$ and $\sqrt{6}$. In this case the image of the Galois group
$$
\Gal \left( \ka \left( \sqrt{2}, \sqrt{3} \right) / \ka \right) \cong \CG_2^2
$$
\noindent in $W(D_5)$ is $\langle \iota_{34}, \iota_{45} \rangle$. The quadruple of $(-1)$-curves $E_1$, $E_5$, $L_{23}$, $E_4$ is defined over $\ka$. Thus one can Galois-equivariantly contract this quadruple and get a del Pezzo surface of degree~$8$. Therefore $X$ is $\ka$-rational by Theorem \ref{ratcrit}. One has $\rho(X)^G = 1$ and the four fixed points of $G$ on $X$ are permuted by the Galois group. Also the surface $Y$ admits a structure of a minimal conic bundle with $K_Y^2 = 2$ by Lemma \ref{DP4i12} and $Y$ is not $\ka$-rational by Theorem \ref{ratcrit}. This gives us an example of a non-$\ka$-rational quotient of a ($G$-minimal) $\ka$-rational del Pezzo surface of degree $4$ by the group $\CG_2$.

If the field $\ka$ contains $\sqrt{2}$ and $\sqrt{3}$ but does not contain $i$ then the image of the Galois group
$$
\Gal \left( \ka \left( i \right) / \ka \right) \cong \CG_2
$$
\noindent in $W(D_5)$ is $\langle \iota_{1345} \rangle$. The surface $X$ admits a structure of a minimal conic bundle by Theorem \ref{MinCB}(iii), since $\rho(X) = 2$. Thus the surface $X$ is not $\ka$-rational by Theorem \ref{ratcrit}. One has $\rho(X)^G = 1$. But on the conic bundle $Y \rightarrow \Pro^1_{\ka}$ all sections with self-intersection $-2$ are defined over $\ka$. Therefore we can Galois-equivariantly contract six components of the singular fibres intersecting one section of the conic bundle $Y \rightarrow \Pro^1_{\ka}$ with selfinterscetion~$-2$ and get a del Pezzo surface of degree $8$ which is $\ka$-rational by Theorem \ref{ratcrit}. This gives us an example of a $\ka$-rational quotient of a $G$-minimal non-$\ka$-rational del Pezzo surface of degree $4$ by the group $\CG_2$.

Now assume that in Example \ref{DP4i12exConst} the field $\ka$ contains $i\sqrt{2}$ and $\sqrt{3}$ but does not contain~$\sqrt{2}$. In this case the image of the Galois group
$$
\Gal \left( \ka \left( \sqrt{2} \right) / \ka \right) \cong \CG_2
$$
\noindent in $W(D_5)$ is $\langle \iota_{34} \rangle$. The pair of $(-1)$-curves $E_1$ and $L_{25}$ is defined over $\ka$. Thus one can Galois-equivariantly contract this pair and get a del Pezzo surface of degree $6$. Therefore $X$ is $\ka$-rational by Theorem \ref{ratcrit}. One has $\rho(X)^G = 2$ therefore $X$ admits a structure of a $G$-minimal conic bundle by Theorem \ref{MinCB}(iii). Note that singular fibres of the conic bundle $Y \rightarrow \Pro^1_{\ka}$ correspond to reducible members of the linear system spanned by the curves $x_1 = 0$ and $x_2 = 0$. One has $K_Y^2 = 2$ therefore the conic bundle $Y \rightarrow \Pro^1_{\ka}$ has six singular fibres. The proper transforms of these six fibres are cut out from $\XX \subset \Pro^4_{\kka}$ by the following six hyperplanes:
$$
x_2 = \pm ix_1, \qquad x_2 = \pm i\sqrt{2} x_1, \qquad x_2 = \pm 2ix_1.
$$
\noindent On the conic bundle $Y \rightarrow \Pro^1_{\ka}$ we can Galois-equivariantly contract four components of the singular fibres corresponding to the hyperplane sections $x_2 = \pm ix_1$, $x_2 = \pm 2ix_1$ and get a del Pezzo surface of degree $6$ which is $\ka$-rational by Theorem \ref{ratcrit}. This gives us an example of a $\ka$-rational quotient of a $G$-minimal $\ka$-rational del Pezzo surface of degree $4$ by the group $\CG_2$.

\end{example}

Now we show that the quotient of a del Pezzo surface of degree $4$ by a group $\VG_4$ can be non-$\ka$-rational.

\begin{example}
\label{DP4i12i13ex}

Let us consider the quotient of the surface $X$ given by equations \eqref{DP4example} by the group $G = \langle\iota_{12}, \iota_{13}\rangle$. The element $\iota_{12}$ has four fixed points $\left( 0 : 0 : \pm\sqrt{3} : \pm\sqrt{2} : 1 \right)$, the element $\iota_{13}$ has four fixed points $\left( 0 : \pm i\sqrt{3} : 0 : \pm i : 1 \right)$ and the element $\iota_{23}$ has four fixed points~$\left( \pm i\sqrt{3} : 0 : 0 : \pm 2i : 1 \right)$.

Suppose that in Example \ref{DP4i12exConst} the field $\ka$ does not contain $i$, $\sqrt{2}$ and~$\sqrt{3}$, and
$$
\Gal \left( \ka \left( i, \sqrt{2}, \sqrt{3} \right) / \ka \right) \cong \CG_2^3.
$$
\noindent As in Example \ref{DP4i12exConst} the sixteen $(-1)$-curves on $X$ are defined over the field $\ka \left( i, \sqrt{2}, \sqrt{3} \right)$. The image of the group $\Gal \left( \ka \left( i, \sqrt{2}, \sqrt{3} \right) / \ka \right)$ in $W(D_5)$ is $\langle \iota_{1345}, \iota_{15}, \iota_{45} \rangle$. In this case $\rho(X) = 2$ and $X$ admits a structure of a minimal conic bundle by Theorem \ref{MinCB}(iii). Thus the surface $X$ is not $\ka$-rational by Theorem \ref{ratcrit}. One has $\rho(X)^G = 1$ and for each nontrivial element in $G$ its four fixed points are permuted by the Galois group. Thus the quotient surface $X / G$ is birationally equivalent to a minimal del Pezzo surface $Y$ of degree $4$ by Lemma \ref{DP4i12i13}. The surface $Y$ is not $\ka$-rational by Theorem \ref{ratcrit}. This gives us an example of a non-$\ka$-rational quotient of a ($G$-minimal) non-$\ka$-rational del Pezzo surface of degree $4$ by the group $\VG_4$.

Now assume that in Example \ref{DP4i12exConst} the field $\ka$ contains $i\sqrt{2}$ but does not contain $\sqrt{2}$, $\sqrt{3}$ and $\sqrt{6}$. In this case the image of the Galois group
$$
\Gal \left( \ka \left( \sqrt{2}, \sqrt{3} \right) / \ka \right) \cong \CG_2^2
$$
\noindent in $W(D_5)$ is~$\langle \iota_{34}, \iota_{45} \rangle$. The quadruple of $(-1)$-curves $E_1$, $E_5$, $L_{23}$, $E_4$ is defined over $\ka$. Thus one can Galois-equivariantly contract this quadruple and get a del Pezzo surface of degree $8$. Therefore $X$ is $\ka$-rational by Theorem \ref{ratcrit}. One has $\rho(X)^G = 1$ and for each nontrivial element in $G$ its four fixed points are permuted by the group $G \times \Gal\left(\kka / \ka \right)$. Thus by Lemma \ref{DP4i12i13} the quotient surface $X / G$ is birationally equivalent to a minimal del Pezzo surface $Y$ of degree $4$ and $Y$ is not $\ka$-rational by Theorem \ref{ratcrit}. This gives us an example of a non-$\ka$-rational quotient of a $G$-minimal $\ka$-rational del Pezzo surface of degree~$4$ by the group~$\VG_4$.

Now assume that in Example \ref{DP4i12exConst} the field $\ka$ contains $i$ and $\sqrt{2}$ but does not contain~$\sqrt{3}$. In this case the image of the Galois group
$$
\Gal \left( \ka \left( \sqrt{3} \right) / \ka \right) \cong \CG_2
$$
\noindent in $W(D_5)$ is $\langle \iota_{45} \rangle$. The pair of $(-1)$-curves $E_1$ and $L_{23}$ is defined over $\ka$. Thus one can Galois-equivariantly contract this pair and get a del Pezzo surface of degree $6$. Therefore $X$ is $\ka$-rational by Theorem \ref{ratcrit}. One has $\rho(X)^G = 1$ but for each nontrivial element in $G$ its four fixed points are not permuted transitively by the Galois group $\Gal\left(\ka\left( \sqrt{3} \right) / \ka \right)$. Thus by Lemma \ref{DP4i12i13} the quotient surface $X / G$ is birationally equivalent to a del Pezzo surface $Y$ of degree $4$ with $\rho(Y) = 4$ and $Y$ is $\ka$-rational by Corollary \ref{piccrit}. This gives us an example of a $\ka$-rational quotient of a $G$-minimal $\ka$-rational del Pezzo surface of degree~$4$ by the group $\VG_4$.

\end{example}

\begin{example}
\label{DP4i12i13exnratrat}

In Example \ref{DP4i12exConst} consider the quotient of $X$ by the group \mbox{$G = \langle\iota_{12}, \iota_{14}\rangle$}. The element $\iota_{12}$ has four fixed points $\left( 0 : 0 : \pm\sqrt{3} : \pm\sqrt{2} : 1 \right)$, the element $\iota_{14}$ has four fixed points $\left( 0 : \pm i\sqrt{2} : \pm 1 : 0 : 1 \right)$ and the element $\iota_{24}$ has four fixed points~$\left( \pm i : 0 : \pm\sqrt{2} : 0 : 1 \right)$.

Suppose that the field $\ka$ contains $\sqrt{2}$ and $\sqrt{3}$ and does not contain $i$. Then the image of the Galois group
$$
\Gal \left( \ka \left( i \right) / \ka \right) \cong \CG_2
$$
\noindent in $W(D_5)$ is $\langle \iota_{1345} \rangle$. The surface $X$ admits a structure of a minimal conic bundle by Theorem \ref{MinCB}(iii), since $\rho(X) = 2$. Thus the surface $X$ is not $\ka$-rational by Theorem \ref{ratcrit}. One has $\rho(X)^G = 1$. But for each nontrivial element in $G$ the four fixed points are not permuted transitively by the Galois group $\Gal\left(\ka\left( i \right) / \ka \right)$. Thus by Lemma \ref{DP4i12i13} the quotient surface $X / G$ is birationally equivalent to a del Pezzo surface $Y$ of degree $4$ with $\rho(Y) = 4$ and $Y$ is $\ka$-rational by Corollary \ref{piccrit}. This gives us an example of a $\ka$-rational quotient of a ($G$-minimal) non-$\ka$-rational del Pezzo surface of degree $4$ by the group $\VG_4$.

\end{example}

Now we show that the quotient of a del Pezzo surface of degree $4$ by a group $\CG_4$ can be non-$\ka$-rational.

\begin{lemma}
\label{DP4C2i15}
Let a finite group $G$ act on a del Pezzo surface $X$ of degree $4$ and
$$
N \cong \CG_4 = \langle (12)(34)\iota_{15} \rangle
$$
\noindent be a normal subgroup in $G$. Then the surface $X / N$ is $G / N$-birationally equivalent to a conic bundle $Y$ with $K_Y^2 = 4$.
\end{lemma}

\begin{proof}

After a suitable change of coordinates the group $N$ is generated by an element
$$
g: (x_1 : x_2 : x_3 : x_4 : x_5) \mapsto (-x_2 : x_1 : x_4 : x_3 : -x_5).
$$
\noindent The element $g$ is of order $4$ and the element $g^2$ has four fixed points $p_1$, $p_2$, $p_3$ and $p_4$ on $X$ cut out by the plane $x_1 = x_2 = 0$. These points are the intersection points of two \mbox{$N$-invariant} conics $C_1$ and $C_2$ in the plane $x_1 = x_2 = 0$. These conics cannot be pointwisely fixed by the group $N$, since $N$ has only four fixed points. Therefore $N$ faithfully acts on $C_1$ and~$C_2$, has two fixed points on $C_1$ and has two fixed points on $C_2$. Thus after relabelling the points $p_i$ if necessary, we have $gp_1 = p_1$, $gp_2 = p_2$, $gp_3 = p_4$ and $gp_4 = p_3$.

Consider a linear system spanned by the curves $x_3 + x_4 = 0$ and $x_5 = 0$. In this linear system there is exactly one member $C$ passing through the points $p_3$ and $p_4$. There are two $g^2$-fixed points on $C$. An element of an automorphism group of order two cannot have two fixed points on an elliptic curve. Therefore $C$ is a singular curve, moreover $C$ has more than one singular point since the element $g$ does not have fixed points on $C$. Thus $C$ is reducible and consists of two smooth conics meeting each other at points $p_3$ and $p_4$.

Let $f: X \rightarrow X / N$ be the quotient map and
$$
\pi: \widetilde{X / N} \rightarrow X / N
$$
\noindent be the minimal resolution of singularities. One can show that $\widetilde{X / N}$ admits a conic bundle structure as in the proof of Lemma \ref{DP4i12}.

The points $f(p_1)$ and $f(p_2)$ are $A_3$ singularities. Exceptional divisors of their resolutions are chains consisting of three $(-2)$-curves each. The point $f(p_3) = f(p_4)$ is an $A_1$~singularity. The curves $\pi^{-1}_*f\left(C_1\right)$, $\pi^{-1}_*f\left(C_2\right)$ and $\pi^{-1}_*f\left( C \right)$ are three disjoint $(-1)$-curves (see Table \ref{table1}).

Let $\sigma: X \rightarrow Y$ be the contraction of the curves $\pi^{-1}_*f\left(C_1\right)$, $\pi^{-1}_*f\left(C_2\right)$ and $\pi^{-1}_*f\left( C \right)$. Then
$$
K_Y^2 = K_{\widetilde{X / N}}^2 + 3 = K_{X / N}^2 + 3 = \frac{1}{4}K_X^2 + 3 = 4
$$
\noindent and $Y$ is a conic bundle.
\end{proof}

\begin{remark}
\label{Isk}
The reducible curves $\pi^{-1}f(p_i)$ are chains of three $(-2)$-curves. The conic bundle $Y$ obtained in Lemma \ref{DP4C2i15} has two sections with self-intersection $-2$ which are transforms of the central $(-2)$-curves in these chains. There exists an elliptic curve $E$ such that the surface~$\overline{Y}$ is birationally equivalent to a quotient of $\Pro^1_{\kka} \times E$ by an involution (see \cite[Subsection 5.2]{DI1} for details). Such a surface is called \textit{Iskovskikh surface}.
\end{remark}

\begin{remark}
\label{DP4C2i15min}
In the notation of Lemma \ref{DP4C2i15} one can check that if $\rho(X)^G = 1$ and the points $p_1$ and $p_2$ are permuted by an element of $G \times \Gal \left( \kka / \ka \right) $ which does not permute the curves $C_1$ and $C_2$ then $\rho(Y)^{G / N} = 2$ and $Y$ is $G / N$-minimal by Theorem \ref{MinCB}(iii).

\end{remark}

Now we construct an explicit example satisfying the conditions of Lemma \ref{DP4C2i15}.

\begin{example}
\label{DP4C2i15ex}
Suppose that the field $\ka$ contains $i$ and does not contain $\sqrt{2}$ and $\sqrt{6}$. Consider a surface~$X$ in $\Pro^4_{\ka}$ given by the equations
$$
4x_1^2 - 4x_2^2 - x_3^2 + x_4^2 = 0, \quad \quad 2x_1^2 + 2x_2^2 - x_3^2 - x_4^2 + 12x_5^2 = 0.
$$
\noindent Note that $X(\ka) \ne \varnothing$ since the $\ka$-point $(1:i:2:2i:0)$ lies on $X$.

The normal subgroup $\CG_2^4$ of $W(D_5)$ acts on $X$ by switching signs of coordinates: elements $\iota_{ij}$ and $\iota_{ijkl}$ switch signs of coordinates $x_i$, $x_j$ and $x_i$, $x_j$, $x_k$, $x_l$, respectively. The group $G \cong \CG_4$ generated by an element
$$
g: (x_1 : x_2 : x_3 : x_4 : x_5) \mapsto (-x_2 : x_1 : x_4 : x_3 : -x_5)
$$
\noindent acts on $X$. The element $g^2$ has four fixed points on $\XX$:
$$
p_1 = (0 : 0 : \sqrt{6} : -\sqrt{6} : 1), \quad p_2 = (0 : 0 : -\sqrt{6} : \sqrt{6} : 1),
$$
$$
p_3 = (0 : 0 : \sqrt{6} : \sqrt{6} : 1), \quad p_4 = ( 0 : 0 : -\sqrt{6} : -\sqrt{6} : 1)
$$

One can check that sixteen $(-1)$-curves on the surface $\XX$ are given by the following parametrization:
$$
\left( \pm \left( x + y \right) : \pm \left( x - y \right) : \pm \sqrt{2}\left( x + 2y \right) : \pm \sqrt{2}\left( x - 2y \right) : \pm y \right)
$$
\noindent These curves are defined over any field containing $\sqrt{2}$ and the image of the group
$$
\Gal \left( \ka \left( \sqrt{2} \right) / \ka \right) \cong \CG_2
$$
\noindent in $W(D_5)$ is $\langle \iota_{34} \rangle$. The pair of $(-1)$-curves $E_3$ and $E_4$ is defined over $\ka$. Thus one can Galois-equivariantly contract this pair and get a del Pezzo surface of degree $6$. Therefore $X$ is $\ka$-rational by Theorem \ref{ratcrit}. One has $\rho(X)^G = 1$ and the points $p_1$ and $p_2$ are permuted by an element of $G \times \Gal \left( \kka / \ka \right) $ which does not permute the curves $C_1$ and $C_2$, given by $x_1 = \pm ix_2$. Therefore by Remark \ref{DP4C2i15min} the quotient surface $X / G$ is birationally equivalent to a minimal conic bundle $Y$ with $K_Y^2 = 4$ and $Y$ is not $\ka$-rational by Theorem~\ref{ratcrit}. This gives us an example of a non-$\ka$-rational quotient of a $G$-minimal $\ka$-rational del Pezzo surface of degree $4$ by the group $\CG_4$.

Now assume that the field $\ka$ contains $\sqrt{6}$ but does not contain $\sqrt{2}$. In this case the image of the Galois group
$$
\Gal \left( \ka \left( \sqrt{2} \right) / \ka \right) \cong \CG_2
$$
\noindent in $W(D_5)$ is $\langle \iota_{34} \rangle$. The pair of $(-1)$-curves $E_3$ and $E_4$ is defined over $\ka$. Thus one can Galois-equivariantly contract this pair and get a del Pezzo surface of degree $6$. Therefore $X$ is $\ka$-rational by Theorem \ref{ratcrit}. One has $\rho(X)^G = 1$ but the points $p_1$ and $p_2$ are not permuted by the Galois group $\Gal\left(\ka\left( \sqrt{2} \right) / \ka \right)$. Thus by Lemma \ref{DP4C2i15} the quotient surface $X / G$ is birationally equivalent to a minimal conic bundle $Y$ with $K_Y^2 = 4$ with $\rho(Y) \geqslant 3$ and $Y$ is $\ka$-rational by Corollary \ref{piccrit}. This gives us an example of a $\ka$-rational quotient of a $G$-minimal $\ka$-rational del Pezzo surface of degree $4$ by the group $\CG_4$.
\end{example}

\begin{example}
\label{DP4C2i15ex2}
Suppose that the field $\ka$ contains $i$ and does not contain $\sqrt{2}$, $\sqrt{3}$ and~$\sqrt{6}$. Consider a surface~$X$ in $\Pro^4_{\ka}$ given by the equations
$$
4x_1^2 - 4x_2^2 -3x_3^2 + 3x_4^2 = 0, \quad \quad 2x_1^2 + 2x_2^2 - 3x_3^2 - 3x_4^2 + 36x_5^2 = 0
$$
\noindent Note that $X(\ka) \ne \varnothing$ since the $\ka$-point $(3:3:0:0:i)$ lies on $X$.

The normal subgroup $\CG_2^4$ of $W(D_5)$ acts on $X$ by switching signs of coordinates: elements $\iota_{ij}$ and $\iota_{ijkl}$ switch signs of coordinates $x_i$, $x_j$ and $x_i$, $x_j$, $x_k$, $x_l$, respectively. The group $G \cong \CG_4$ generated by an element
$$
g: (x_1 : x_2 : x_3 : x_4 : x_5) \mapsto (-x_2 : x_1 : x_4 : x_3 : -x_5)
$$
\noindent acts on $X$. The element $g^2$ has four fixed points on $\XX$:
$$
p_1 = (0 : 0 : \sqrt{6} : -\sqrt{6} : 1), \quad p_2 = (0 : 0 : -\sqrt{6} : \sqrt{6} : 1),
$$
$$
p_3 = (0 : 0 : \sqrt{6} : \sqrt{6} : 1), \quad p_4 = ( 0 : 0 : -\sqrt{6} : -\sqrt{6} : 1)
$$

One can check that sixteen $(-1)$-curves on the surface $\XX$ are given by the following parametrization:
$$
\left( \pm \sqrt{3}\left( x + y \right) : \pm \sqrt{3}\left( x - y \right) : \pm \sqrt{2}\left( x + 2y \right) : \pm \sqrt{2}\left( x - 2y \right) : \pm y \right)
$$
\noindent These curves are defined over any field containing $\sqrt{2}$ and $\sqrt{3}$ and the image of the group
$$
\Gal \left( \ka \left( \sqrt{2}, \sqrt{3} \right) / \ka \right) \cong \CG_2^2
$$
\noindent in $W(D_5)$ is $\langle \iota_{12}, \iota_{34} \rangle$. In this case $\rho(X) = 2$ and $X$ admits a structure of a minimal conic bundle by Theorem \ref{MinCB}(iii). Thus the surface $X$ is not $\ka$-rational by Theorem \ref{ratcrit}. One has $\rho(X)^G = 1$ and the points $p_1$ and $p_2$ are permuted by an element of $G \times \Gal \left( \kka / \ka \right) $ which does not permute the curves $C_1$ and $C_2$, given by $x_1 = \pm ix_2$. Therefore by Remark~\ref{DP4C2i15min} the quotient surface $X / G$ is birationally equivalent to a minimal conic bundle $Y$ with $K_Y^2 = 4$ and $Y$ is not $\ka$-rational by Theorem \ref{ratcrit}. This gives us an example of a non-$\ka$-rational quotient of a ($G$-minimal) non-$\ka$-rational del Pezzo surface of degree $4$ by the group $\CG_4$.

Now assume that the field $\ka$ contains $\sqrt{6}$ but does not contain $\sqrt{2}$ and $\sqrt{3}$. In this case the image of the Galois group
$$
\Gal \left( \ka \left( \sqrt{2} \right) / \ka \right) \cong \CG_2
$$
\noindent in $W(D_5)$ is $\langle \iota_{1234} \rangle$. In this case $\rho(X) = 2$ and $X$ admits a structure of a minimal conic bundle by Theorem \ref{MinCB}(iii). Thus the surface $X$ is not $\ka$-rational by Theorem \ref{ratcrit}. One has $\rho(X)^G = 1$ but the points $p_1$ and $p_2$ are not permuted by the Galois group~$\Gal\left(\ka\left( \sqrt{2} \right) / \ka \right)$. Thus by Lemma \ref{DP4C2i15} the quotient surface $X / G$ is birationally equivalent to a minimal conic bundle $Y$ with $K_Y^2 = 4$ with $\rho(Y) \geqslant 3$ and $Y$ is $\ka$-rational by Corollary \ref{piccrit}. This gives us an example of a $\ka$-rational quotient of a ($G$-minimal) non-$\ka$-rational del Pezzo surface of degree $4$ by the group $\CG_4$.
\end{example}

\bibliographystyle{alpha}
\bibliography{my_ref}

\def\cprime{$'$}
\begin{thebibliography}{XX}

\bibitem[AN06]{AN06}
V.\,A.\,Alekseev, V.\,V.\,Nikulin,
\newblock Del Pezzo and K3 surfaces,
\newblock MSJ Memoirs, 15, Mathematical Society of Japan, Tokyo, 2006

\bibitem[Dol12]{Dol12}
I.\,V.\,Dolgachev,
\newblock Classical algebraic geometry: a modern view,
\newblock Cambridge University Press, Cambridge, 2012

\bibitem[DI09]{DI1}
I.\,V.\,Dolgachev, V.\,A.\,Iskovskikh,
\newblock Finite subgroups of the plane Cremona group,
\newblock In: Algebra, arithmetic, and geometry, vol. I: In Honor of Yu.\,I.\,Manin, Progr. Math., 269, 443--548, Birkh\"auser, Basel, 2009

\bibitem[Isk79]{Isk79}
V.\,A.\,Iskovskikh,
\newblock Minimal models of rational surfaces over arbitrary field,
\newblock Math. USSR Izv., 1980, 14(1), 17--39

\bibitem[Isk96]{Isk96}
V.\,A.\,Iskovskikh,
\newblock Factorization of birational mappings of rational surfaces from the point of view of Mori theory,
\newblock Russian Math. Surveys, 1996, 51, 585--652

\bibitem[Man67]{Man67}
Ju.\,I.\,Manin,
\newblock Rational surfaces over perfect fields. II,
\newblock Math. Sb., 1967, 1(2), 141--168

\bibitem[Man74]{Man74}
Yu.\,I.\,Manin.
\newblock Cubic forms: algebra, geometry, arithmetic,
\newblock In: North-Holland Mathematical Library, Vol. 4, North-Holland Publishing Co., Amsterdam-London; American Elsevier Publishing Co., New York, 1974

\bibitem[Pop14]{Pop14}
V.\,L.\,Popov,
\newblock Jordan groups and automorphism groups of algebraic varieties,
\newblock Springer Proceedings in Mathematics \& Statistics, 2014, 79, Automorphisms in Birational and Affine Geometry, 185--213

\bibitem[Tr14]{Tr14}
A.\,S.\,Trepalin,
\newblock Rationality of the quotient of $\mathbb{P}^2$ by finite group of automorphisms over arbitrary field of characteristic zero,
\newblock Cent. Eur. J. Math., 2014, 12(2), 229--239

\bibitem[Tr16]{Tr16}
A.\,S.\,Trepalin,
\newblock Quotients of conic bundles,
\newblock Transformation Groups, 2016, 21(1), 275--295

\end{thebibliography}
\end{document}